\documentclass[preprint]{imsart}

\pdfoutput=1

\usepackage[linesnumbered,ruled]{algorithm2e}
\RestyleAlgo{ruled}

\usepackage{xcolor}
\usepackage{mathrsfs,mathtools,latexsym, enumerate, amsmath, amsthm, amssymb,hyperref,amsfonts, pdftricks,tikz,framed,color,enumerate,xcolor,bbm,appendix,soul,enumitem}
\usepackage{romanbar}
\usepackage{MnSymbol} 
\usepackage{algorithmic}
\usepackage{multirow,mathtools,latexsym, mathrsfs, pdftricks,framed,listings,mleftright}

\usepackage{romanbar}
\usepackage{todonotes}

\newcommand{\norm}[1]{\left\lVert#1\right\rVert}

\newcommand{\dd}{\, \text{d}}

\theoremstyle{remark}
\newtheorem{definition}{Definition}
\newtheorem{remark}{Remark}

\theoremstyle{plain}% default
\newtheorem{prototheorem}{Theorem}

\newtheorem{theorem}[prototheorem]{Theorem}

\newtheorem{lemma}[prototheorem]{Lemma}
\newtheorem{corollary}[prototheorem]{Corollary}

\newtheorem*{assumption*}{\assumptionnumber}
\providecommand{\assumptionnumber}{}
\makeatletter
\newenvironment{assumption}[1]
 {%
  \renewcommand{\assumptionnumber}{Assumption A.$#1$}%
  \begin{assumption*}%
  \protected@edef\@currentlabel{A.#1}%
 }
 {%
  \end{assumption*}
 }

\makeatother

\newcommand{\eps}{\varepsilon}

\newcommand{\mnorm}[1]{\left\| #1 \right\|_{\mathsf{op}}} 
\newcommand{\tw}{\mathsf{tw}}
\newcommand{\twnorm}[1]{\left\| #1 \right\|_{\mathsf{tw}}} 
\newcommand{\W}{\boldsymbol{\mathcal{W}}}
\newcommand{\PM}{\ensuremath \mathcal{P}}
\newcommand{\J}{\mathrm{Couplings}}
\newcommand{\law}{\operatorname{Law}}

\newcommand{\TV}{\mathsf{TV}}

\newcommand{\tr}{\operatorname{tr}}

\newcommand{\tmix}{\tau_{\operatorname{mix}}}

\newcommand{\rn}[1]{\Romanbar{#1}}

\newcommand{\PP}{\Pi}
\newcommand{\ind}{\mathbf{1}}
\newcommand{\diam}{\operatorname{diam}}
\newcommand{\op}{\mathsf{op}}

\usepackage[backend=biber,backref=true,style=numeric,firstinits=true,maxbibnames=9]{biblatex}
\begin{document}

\begin{frontmatter}
\title{Mixing of Metropolis-Adjusted  \\ Markov Chains via Couplings:\\ The High Acceptance Regime}
\runtitle{Mixing of Metropolis-Adjusted Markov Chains}

\begin{abstract} 
We present a coupling framework to upper bound the total variation mixing time of various Metropolis-adjusted, gradient-based Markov kernels in the `high acceptance regime'.
The approach uses a localization argument to boost local mixing of the underlying unadjusted kernel to mixing of the adjusted kernel when the acceptance rate is suitably high. As an application, mixing time guarantees are developed for a non-reversible, adjusted Markov chain based on the kinetic Langevin diffusion, where little is currently understood. 
\end{abstract}

\begin{aug}
\Author[A]{\fnms{Nawaf}~\snm{Bou-Rabee}\ead[label=e1]{nawaf.bourabee@rutgers.edu}}
%,
%\Author[B]{\fnms{Andreas}~\snm{Eberle}\ead[label=e2]{eberle@uni-bonn.de}}
\and
\Author[B]{\fnms{Stefan}~\snm{Oberd\"orster}\ead[label=e3]{oberdoerster@uni-bonn.de}}
\address[A]{Department of Mathematical Sciences \\ Rutgers University Camden \\ 311 N 5th Street \\ Camden, NJ 08102 USA \\ \href{mailto:nawaf.bourabee@rutgers.edu}{nawaf.bourabee@rutgers.edu}}

\address[B]{Institute for Applied Mathematics \\ University of Bonn \\ Endenicher Allee 60 \\ D-53115 Bonn Germany \\
%\href{mailto:eberle@uni-bonn.de}{eberle@uni-bonn.de} \\
\href{mailto:oberdoerster@uni-bonn.de}{oberdoerster@uni-bonn.de}}
\runAuthor{N.~Bou-Rabee,
%A.~Eberle, 
S.~Oberd\"orster}
\end{aug}


% 60J05: Discrete-time Markov processes on general state spaces
%
% 60J60: Diffusion processes
% 60J25: Continuous-time Markov processes on general state spaces

\end{frontmatter}

\maketitle

\section{Introduction}

A nearly universal ingredient to gradient-based  Markov chain Monte Carlo (MCMC) kernels are time discretizations of measure-preserving SDEs or PDMPs such as the kinetic Langevin diffusion and Andersen dynamics \cite{RoTw1996B,DiHoNe2000,MaStHi2002,ELi2008,BoSaActaN2018,EbMa2019,Deligiannidis2021,kleppe2022,BoEb2022}. These kernels are gradient-based in the sense that they incorporate and rely on evaluation of the gradient of the log-density of the target distribution. In practice, the asymptotic bias due to time discretization is either incurred (leading to \emph{unadjusted} kernels) or eliminated by a Metropolis-Hastings filter (leading to \emph{adjusted} kernels). In either case, a question that is both fundamental mathematically and crucial to applications is \cite{diaconis1995we,meyntweedie1993,Di2009,levin2009markov,villani2008optimal,MontenegroTetali,douc2018markov}: 
\emph{Starting from a distribution $\nu$, how many steps $n \in \mathbb{N}$ are sufficient for the $n$-step distribution of the Markov chain to be an $\varepsilon$-accurate approximation of the stationary distribution in total variation?}  The smallest such number of steps is the so-called $\varepsilon$-\emph{mixing time} of the Markov chain from the initial distribution $\nu$.

Recently, there has been significant progress in quantifying the mixing time of unadjusted, gradient-based kernels including the unadjusted Langevin algorithm \cite{durmus2017nonasymptotic,dalalyan2017theoretical,durmus2019high,erdogdu2021convergence},
unadjusted HMC \cite{BouRabeeSchuh2023,BouRabeeEberle2023,monmarche2022hmc}, and various unadjusted chains based on the kinetic Langevin diffusion \cite{cheng2018underdamped,dalalyan2020sampling, monmarche2021high,monmarche2022hmc}; see \cite{DurmusEberle2021} for a unified and comprehensive treatment of unadjusted MCMC methods.  These works give explicit upper bounds on the mixing time and complexity, which reveal that the time step size required to adequately resolve the asymptotic bias depends substantially on the accuracy  $\varepsilon$. This potentially costly dependence motivates Metropolis adjustment, which eliminates the asymptotic bias by employing a Metropolis-Hastings filter. Intuitively speaking, it ensures the proportion of steps the adjusted chain spends in a given region equals the measure of that region with respect to the stationary distribution \cite{MeRoRoTeTe1953,Ha1970, diaconis1995we,Tierney1998,BiDi2001,Di2009,andrieu2020general}.  As a consequence, though, the adjusted chain involves a complex interplay between the transition step of the unadjusted kernel and the stationary distribution; to quote  Bilera \& Diaconis [2001], \emph{``for many people ... the Metropolis-Hastings algorithm seems like a magic trick. It is hard to see where it comes from or why it works.''} Needless to say, the mixing time analysis of adjusted kernels is mathematically more delicate than of unadjusted kernels.

Intrinsically capturing the interplay described above, the notion of conductance has played a significant role in quantifying the mixing time of adjusted kernels.
Classical conductance arguments are commonly used to identify bottlenecks, which yield mixing time lower bounds \cite{LST21,LST20,chewi2020optimal}.
%For such mixing time lower bounds, it should be emphasized that reversibility is not required.
For adjusted kernels that in addition are reversible, conductance arguments can be adapted to obtain mixing time upper bounds; see, e.g., for MALA and HMC \cite{chewi2020optimal,chen22,chen23}.
While these works make mild assumptions on the stationary distribution (e.g.~isoperimetric inequalities) and often yield sharp mixing time upper bounds, a warm start assumption is inevitable.
%since the notion of conductance involves the equilibrium flow of the chain.
In particular, these mixing time upper bounds typically depend logarithmically on the $L^{\infty}$-norm of the relative density of the initial to the stationary distribution; see \cite{chen2020fast,LST20} for progress towards double-logarithmic dependence
and \cite{alt23} for sampling from a warm start.
%In general, it is unclear how many steps of the chain are needed to sample from an initial distribution with an  $L^{\infty}$-bounded  relative density that does not deteriorate exponentially with the underlying parameters.
%See however the recent work \cite{alt23} for promising steps towards sampling from a warm start for MALA.
Apart from conductance arguments, current mathematical tools are limited in their ability to obtain quantitative mixing time upper bounds for non-reversible adjusted kernels, even from warm starting distributions. 
%Conductance based methods have also produced results from less demanding notions of warmness, cf. \cite{apers2022hamiltonian}, and from a feasible start in a given normal distribution, cf. \cite{LST20,chen2020fast} and \cite{LST21} for lower bounds.

As a step towards filling the gap in capability outlined above, in this work we introduce a new coupling framework to obtain mixing time guarantees for Metropolis-adjusted, gradient-based Markov chains.  Let $\varepsilon>0$ be the desired total variation ($\TV$) accuracy. The underlying idea is to fix an epoch $\mathfrak E > 0$ of steps such that two copies of the unadjusted chain given by the kernel $\pi^u$ starting from different initial conditions $x$ and $\tilde{x}$ meet with probability at least $1-(3 e)^{-1}$ after $\mathfrak E$ steps, i.e., 
\begin{equation}
\label{intro:ucoupling}
\|\delta_x(\pi^u)^{\mathfrak E}-\delta_{\tilde x}(\pi^u)^{\mathfrak E}\|_{\TV}\ \leq\ (3e)^{-1} \;, \end{equation}
where we used the coupling characterization of the $\TV$ distance $\| \cdot \|_{\TV}$.
A standard way to ensure \eqref{intro:ucoupling} is to use a contractive coupling for $\mathfrak E-1$
steps, followed by a one-shot coupling  \cite{roberts2002one,madras2010quantitative,EbMa2019,monmarche2021high, BouRabeeEberle2023}. The time step size is then tuned such that the probability of  a rejection occurring in this epoch is at most $2(3 e)^{-1}$, and crucially, this tuning is at most logarithmic in $1/\varepsilon$.  
Hence, after one epoch, the adjusted kernel $\pi$ satisfies 
\begin{equation}\label{intro:acoupling}
\|\delta_x\pi^{\mathfrak E}-\delta_{\tilde x}\pi^{\mathfrak E}\|_{\TV}\ \leq\ e^{-1} \;.
\end{equation}
Therefore, after $\lceil \log(1/\eps) \rceil$ epochs, it immediately follows that there exists a coupling of the adjusted kernel which meets with probability at least $1-\eps$.  
Iterating the epochs is an important step in this new coupling approach, and without this iteration, as Monmarch{\'e} noted in \cite{monmarche2021high}, the aforementioned proof fails to capture a logarithmic scaling of the mixing time with respect to $1/\varepsilon$.

Stated precisely in Theorem \ref{thm:mix}, the main result of this paper provides a  broadly applicable coupling framework to obtain mixing time upper bounds for Metropolis-adjusted, gradient-based Markov chains without imposing restrictive assumptions on either the stationary or the starting distribution. In essence, the theorem uses a localization argument to boost local mixing of the unadjusted kernel to mixing of the adjusted kernel when the Metropolis filter intervenes over each epoch with sufficiently low probability, i.e., in the \emph{high acceptance regime}: a notion that is made precise in \S\ref{sec:HAR}.
The \emph{low acceptance regime}, allowing for more frequent rejection, falls beyond the scope of this work. As a nontrivial application of Theorem~\ref{thm:mix}, in \S\ref{sec:MAKLA} we develop mixing time upper bounds for a non-reversible, adjusted Markov chain based on the kinetic Langevin diffusion.

\subsection*{Complimentary Literature}

 Here we briefly highlight some complimentary literature on related but different probabilistic techniques for mixing time analysis. In recent years, there has been  progress in developing couplings for a variety of Metropolis-adjusted, gradient-based chains whose stationary distributions display high-dimensionality and/or non-logconcavity.  In particular, dimension-free upper bounds in Wasserstein distance have been developed for a variant of MALA suitable for perturbations of Gaussian measures in high dimensions \cite{Eb2014}.  Moreover, a coupling of adjusted HMC that is contractive  in non-logconcave settings was introduced in \cite{BoEbZi2020}; this coupling offers  flexibility for extensions/applications \cite{heng2019unbiased,BoEb2020,BouRabeeSchuh2023}. For MALA and related Markov chains, coupling-based contractivity results are also available in distances that interpolate between  $L^1$-Wasserstein and $\TV$ \cite{EbMa2019}.  Moreover, a variety of couplings tailored to Metropolis-Hastings kernels, including maximal couplings, have recently been proposed for MCMC convergence analysis in high dimensions \cite{heng2019unbiased,jacob2020unbiased,wang2021maximal,o2021metropolis}. In addition, there is a considerable and growing body of work devoted to Harris Ergodic Theorem, which is a very powerful tool for verifying geometric ergodicity of Markov chains  \cite{meyntweedie1993,MeTw1994,rosenthal1995minorization,RoRo2004,hairer2010convergence,DuMoSa2020}; for a simple and elegant proof see \cite{hairer2011yet}.  Over the years there have been many successful applications of this tool including \cite{MenTw1996,RoTw1996A,MaStHi2002,Ta2002,BoHa2013,hairer2014spectral,Bu2014,DuFoMo2016,livingstone2019,DuMoSa2020}, just to cite a few.  There have also been significant advances in refining Harris Ergodic Theorem to obtain more explicit quantitative bounds under more easily verifiable conditions \cite{hairer2011asymptotic,eberle2019quantitative,de2019convergence,durmus2022geometric,YaRo2023}.

\subsection*{Acknowledgements}

The authors are greatly indebted to Andreas Eberle for his guidance \& support. The work of N.~Bou-Rabee has been partially supported by the National Science Foundation under Grant No.~DMS-2111224.  The work of S.~Oberd\"orster has been funded by the Deutsche Forschungsgemeinschaft (DFG, German Research Foundation) under Germany’s Excellence Strategy – EXC-2047/1 – 390685813.

\section{Main Result}\label{sec:mix}

Let $(\Omega,\mathcal A,\mathbb P)$ be a probability space and let $S$ be a Polish state space with metric $d$ and Borel $\sigma$-algebra $\mathcal B$.
Denote by $\mathcal P(S)$ the set of probability distributions on $(S,\mathcal B)$.
Let $\mu\in\mathcal P(S)$.  A standard way to construct a gradient-based, ergodic Markov chain with stationary distribution $\mu$ is to first construct a $\mu$-preserving, ergodic Markov chain with transition kernel $\pi^{exact}$ from the exact flow of a $\mu$-preserving SDE or PDMP.
%Examples include the continuous time overdamped and kinetic Langevin dynamics based on their respective SDEs, as well as exact Hamiltonian Monte Carlo with full and partial velocity refreshment based on the Hamiltonian flow ODE combined with an Ornstein-Uhlenbeck part that introduces the noise necessary for ergodicity.
%Note that $\pi^{exact}$ can hence be a continuous time or a discrete time process.
Both for theoretical purposes and for implementability in applications, it can be desirable to replace the exact flow in $\pi^{exact}$ by an approximate flow based on time-discretization, which yields an unadjusted Markov transition kernel $\pi^u$.
%Since these exact flows are hardly ever available, approximate flows based on time-discretization are often considered.
%Generally speaking, this is done by separating the Hamiltonian flow from the Ornstein-Uhlenbeck part as is already the case in HMC and subsequently discretizing the Hamiltonian flow.
%The step size in the latter will be denoted as $h$.
%Overdamped Langevin fits this picture since it can be approximated by HMC with full velocity refreshment and suitably short integration time of the Hamiltonian dynamics.
%By simply replacing the exact flow in $\pi^{exact}$ with such an approximate flow, one obtains an unadjusted Markov transition kernel $\pi^u$.
However, this unadjusted kernel has the significant drawback that $\mu \pi^u \ne \mu$. Resolving the resulting asymptotic bias in applications can be infeasible. Metropolis-adjustment provides a tool for correcting the stationary distribution and produces an adjusted transition kernel $\pi$ satisfying $\mu \pi = \mu$.
More precisely, we consider transition steps $X\sim\pi(x,\cdot)$ that for $\omega\in\Omega$ are of the general form
\begin{equation} \label{eq:X}
X(\omega)\ =\ \Phi(\omega,x)\,\ind_{A(x)}(\omega)\ +\ \Psi(\omega,x)\,\ind_{A(x)^c}(\omega)\;, \end{equation}
where $\Phi,\Psi:\Omega\times S\to S$ are product measurable and such that $\Phi(\cdot,x)\sim\pi^u(x,\cdot)$ and $\Psi(\cdot,x)\sim\pi^r(x,\cdot)$ for all $x\in S$, where $\pi^r$, like $\pi^u$, is a probability kernel on $(S,\mathcal B)$.
Hereafter, we omit $\omega$ from the notation, writing $\Phi(\cdot,x)=\Phi(x)$ and $\Psi(\cdot,x)=\Psi(x)$.
The indicator function of the event $A(x)=\{\mathcal U\leq\alpha(x,\Phi(x))\}\subseteq\Omega$ with an independent $\mathcal U\sim\mathrm{Unif}(0,1)$ indicates that the proposal $\Phi(x)$ is accepted.  Otherwise the proposal is rejected, in which case the chain is allowed to move according to $\Psi$.

\subsection{The High Acceptance Regime}

\label{sec:HAR}

We now introduce the \emph{high acceptance regime}, in which acceptance occurs sufficiently often such that the adjusted kernel inherits mixing properties of the exact kernel via the unadjusted kernel.   
%Since $\pi^u$ usually converges to its stationary distribution in a similar way as $\pi^{exact}$ converges to $\mu$, we can then relate the mixing of the exact process to the mixing of the Metropolis-adjusted chain.
%For a unified notion of time throughout the considered processes, let $\pi^{exact}$ be given as a discrete time chain the discretization of which directly yields $\pi^u$.
%For example, if $\pi^u$ is the transition kernel of ULA with step size $h$, $\pi^{exact}(x,\cdot)$ is the law of the solution of the overdamped Langevin SDE with initial value $x$ at time $h$.
The $\TV$-mixing time of $\pi^{exact}$ started in the distribution $\eta\in\mathcal P(S)$ to a specified accuracy $\delta>0$ is defined by
\begin{equation}\label{tmixexact}
    \tmix^{exact}(\delta,\eta)\ =\ \inf\bigr\{n\geq0\,:\,\|\eta(\pi^{exact})^n-\mu\|_{\TV}\leq\delta\bigr\}\;.
\end{equation}
The \emph{high acceptance regime} is characterized by the acceptance rate being suitably controlled over a time scale set by the mixing time of the exact kernel which, in turn, will yield a mixing time upper bound for the adjusted kernel by comparison.
%To relate the mixing of $\pi^{exact}$ to $\pi$, we need to compare the time scales of both kernels given by the physical time scale of the underlying SDE or ODE.  If $\pi^{exact}$ is a discrete time kernel, it integrates the underlying SDE or ODE over a fixed interval the length of which we set to be the time scale $T^{exact}$ of $\pi^{exact}$.  If on the other hand $\pi^{exact}$ is in continuous time, we set $T^{exact}=1$.  The time scale $T$ of $\pi$ equals the one of $\pi^u$ and is defined, similar to $T^{exact}$ in the discrete time case, as the integration time of the underlying dynamics.
\begin{definition} \label{def:HAR}
On a collection $\mathcal C\subseteq\mathcal P(S)$ such that $\{\eta\pi:\eta\in\mathcal C\}\subseteq\mathcal C$, $\pi$ is in the \emph{high acceptance regime}, if
\begin{equation}\label{HAR}
    \sup_{\eta\in\mathcal C}\tmix^{exact}\bigr((3e)^{-1},\eta\bigr)\:\cdot\:\sup_{\eta\in\mathcal C}\mathbb P_{x\sim\eta}\bigr(A(x)^c\bigr)\ \leq\ (3e)^{-1}\;.
\end{equation}
\end{definition}

\medskip

A key feature of Definition~\ref{def:HAR} is that the restrictiveness of the condition \eqref{HAR} strongly depends on the choice of $\mathcal C$: the larger the collection $\mathcal C$, the more restrictive \eqref{HAR} becomes.  In one extreme $\mathcal C=\{\mu\}$, the adjusted kernel is \emph{always} in the high acceptance regime since the left hand side of \eqref{HAR} trivially vanishes.  This work is concerned with the other extreme: cold start distributions corresponding to $\mathcal C$ including distributions which may not even be absolutely continuous with respect to $\mu$.  This feature of the definition is what motivates formulating the high acceptance regime in terms of $\pi^{exact}$.

%For $\mathcal C=\{\nu\pi^n:n\geq0\}$, \eqref{HAR} ensures the probability of intervention by the Metropolis filter over any epoch of length $\tmix^{exact}((3e)^{-1},\nu\pi^n)$ to be controlled.
%This allows us to reduce the adjusted chain during these epochs to the the unadjusted chain which in turn
%If unadjusted chain converges at a similar pace as the exact chain, this suffices to show a minorization condition for the adjusted chain.

\subsection{Mixing in the High Acceptance Regime}

Assumption \ref{A_S} stated below is geared towards the high acceptance regime defined in Definition~\ref{def:HAR} with  $\mathcal{C}$ including cold start distributions.   Under Assumption \ref{A_S}, Theorem~\ref{thm:mix} gives mixing time upper bounds for the adjusted kernel.  To better understand Assumption \ref{A_S}, a brief description is provided.

\smallskip

The possibility of cold start distributions motivates using pointwise acceptance probability bounds for the adjusted chain. However, since such bounds often degenerate at infinity, \ref{A_S} {\it(iv)} is introduced to localize the adjusted chain to a bounded domain $D \subseteq S$ with sufficiently high probability.  By association, the underlying unadjusted chain is similarly localized to $D$.

In this domain, and intuitively speaking, \ref{A_S} {\it(i)} and {\it(ii)} require that the underlying unadjusted kernel admits a \emph{locally} successful coupling.  More precisely, \ref{A_S} {\it(i)} assumes there exists a coupling for $\pi^u$ that is \emph{locally} contractive in $D$; and \ref{A_S} {\it(ii)} assumes there exists a \emph{local} one-shot coupling for $\pi^u$ in $D$.

Although stated in a slightly different way, the main idea underlying \ref{A_S} {\it(iii)} is \eqref{HAR}.  Indeed, the epoch $\mathfrak E$ of transition steps appearing in {\it(iii)} is defined in such a way that by {\it(i)} and {\it(ii)}, there exists a coupling of two copies of the unadjusted chain starting at two different initial conditions within $D$ that induces meeting with probability at least $1-(3 e)^{-1}$; therefore, this epoch $\mathfrak E$ is analogous to   $\sup_{\eta\in\mathcal C}\tmix^{exact}\bigr((3e)^{-1},\eta\bigr)$ in \eqref{HAR}.

Denote by $\Delta$ the diagonal in the product space $S\times S$.
The couplings appearing in Assumption~\ref{A_S} are all assumed to be \emph{faithful}.  Recall that a coupling $\Pi$ is faithful if  $\Pi((x,x),\Delta)=1$ for all $x\in S$.
Couplings of the adjusted kernel inherit this property from couplings of the unadjusted kernel if a synchronous coupling of the underlying uniform random variables in the Metropolis filter is used.

Similarly to \eqref{tmixexact}, define the $\TV$-mixing time of the adjusted kernel with initial distribution $\eta\in\mathcal P(S)$ and accuracy $\delta>0$ to be
\begin{equation}\label{eq:tmix}
    \tmix(\delta,\eta)\ =\ \inf\bigr\{n\geq0\,:\,\|\eta\pi^n-\mu\|_{\TV}\leq\delta\bigr\}\;.
\end{equation}

\smallskip

We are now prepared to state Assumption \ref{A_S} and then immediately afterwards the main result of the paper, followed by its proof.

\begin{assumption}{{\bf S}}\label{A_S}
Let $\eps>0$ be the accuracy, $\nu\in\mathcal P(S)$ be the initial distribution, and $D\subseteq S$ be a domain such that $\diam_d(D)\leq R$ for some $R>0$.

\medskip

\noindent
Regarding the unadjusted transition kernel, we require:

\smallskip

\begin{itemize}
\item[(i)]
There exists $\rho>0$ and for all $x,\tilde{x}\in D$ a coupling $\Pi^u_{Contr}((x,\tilde{x}),\cdot)$ of $\pi^u(x,\cdot)$ and $\pi^u(\tilde{x},\cdot)$ such that the contractivity
\begin{equation*}\label{eq:contrass}
    \mathbb{E}d(X^u,\widetilde{X}^u)\ \leq\ (1-\rho)d(x,\tilde{x})
\end{equation*}
holds for $(X^u,\widetilde{X}^u)\sim\Pi^u_{Contr}((x,\tilde{x}),\cdot)$.

\item[(ii)]
There exists $C_{Reg}>0$ and for all $x,\tilde{x}\in D$ a coupling $\Pi^u_{Reg}((x,\tilde{x}),\cdot)$ of $\pi^u(x,\cdot)$ and $\pi^u(\tilde{x},\cdot)$ satisfying the regularization
\begin{equation*}\label{eq:OSass}
    \Pi^u_{Reg}\bigr((x,\tilde{x}),\Delta^c\bigr)\ \leq\ C_{Reg}d(x,\tilde{x})\;.
\end{equation*}

\end{itemize}

\smallskip

\noindent
Regarding the adjusted transition kernel, we require: 

\smallskip

\begin{itemize}

\item[(iii)]
Set the length of an epoch of transition steps at
\[ \mathfrak E\ =\ \bigr\lceil\rho^{-1}\log(3eC_{Reg}R)\bigr\rceil+1 \]
and suppose
\begin{equation*}\label{eq:rejprobass}
    \mathfrak E\ \sup_{x\in D}\mathbb{P}(A(x)^c)\ \leq\ (3e)^{-1}\;.
\end{equation*}
%This enables us to compare the adjusted algorithm to the unadjusted one by estimating the probability of intervention by the Metropolis filter over the length of each epoch.

\medskip

\item[(iv)]
To reduce to the local properties fixed hitherto, we require control of the exit probability from $D$ over the total number of transition steps
\[ \mathfrak H\ =\ \mathfrak E\:\lceil\log(2/\eps)\rceil \]
consisting of sufficiently many epochs to conclude mixing to $\eps$ accuracy.
Therefore let $T=\inf\{k\geq0\,:\,X_k\notin D\}$ and presume
\begin{equation*}\label{eq:exitprob}
    \mathbb P\bigr(T\:\leq\:\mathfrak H\bigr)\ \leq\ \eps/4
\end{equation*}
both for $X_0\sim\nu$ and $X_0\sim\mu$.
\end{itemize}
\end{assumption}

\medskip

\begin{theorem}\label{thm:mix}
    Suppose Assumption~\ref{A_S} holds for $\eps>0$ and $\nu\in\mathcal P(S)$.
    Then
    \[ \tmix(\eps,\nu)\ \leq\ \mathfrak H\;. \]
\end{theorem}

\medskip

\begin{proof}
On the same probability space, consider two copies of the adjusted chain $X_n\sim\nu\pi^n$ and $\widetilde X_n\sim\mu\pi^n=\mu$, one of which in stationarity.
Denote by $T$, $\widetilde T$ the first exit times from $D$ of $X_n$ and $\widetilde X_n$ respectively, and let $\mathfrak T=\min(T,\widetilde T)$. It is notationally convenient to introduce the epoch $m+1=\mathfrak E$ of transition steps and the total number of epochs  $k=\lceil\log(2/\eps)\rceil$ that will be needed to attain $\eps$ accuracy.  Thus, the total number of transitions to reach the desired accuracy will be $k (m+1)=\mathfrak H$.

\smallskip

To see that $k (m+1)$ transition steps of the adjusted chain do indeed suffice, below we will use Assumption~\ref{A_S} {\it(i)} - {\it(iii)} to prove that over each epoch the following bound holds for all $x,\tilde x\in D$:
\begin{equation}\label{minor}
    \mathbb{P}_{(x,\tilde x)}\Bigr(\{X_{m+1}\ne\widetilde{X}_{m+1}\}\cap\bigcap_{l=0}^m\{X_l,\widetilde X_l\in D\}\Bigr)\ \leq\ e^{-1}\;,
\end{equation}
where $\mathbb P_{(x,\tilde x)}$ is the distribution conditioned on $X_0=x$ and $\widetilde X_0=\tilde x$.
Iterating \eqref{minor} $k$ times  will then yield the desired $\TV$-convergence to $\eps$-accuracy. Indeed, by the coupling characterization of the $\TV$-distance, note that the $\TV$-distance to stationarity after $k(m+1)$ transition steps satisfies 
\begin{align}\label{eq:mixover}
        &\|\nu\pi^{k(m+1)}-\mu\|_{TV}
\ \leq\  \mathbb P\bigr(X_{k(m+1)}\ne\widetilde X_{k(m+1)}\bigr) \nonumber\\
& \qquad \leq\  \mathbb P\bigr(X_{k(m+1)}\ne\widetilde X_{k(m+1)},\mathfrak T\geq k(m+1)\bigr)\ +\ \mathbb P\bigr(\mathfrak T\leq k(m+1)\bigr)\;.  
\end{align} The second term  in \eqref{eq:mixover} describes the probability that at least one copy exits $D$ within $k (m+1)$ transition steps, and  by Assumption~\ref{A_S} {\it(iv)} satisfies
\[ \mathbb P\bigr(\mathfrak T\leq k(m+1)\bigr)
\ \leq\  \mathbb P\bigr(T\leq k(m+1)\bigr)+\mathbb P\bigr(\widetilde T\leq k(m+1)\bigr)
\ \overset{{\it(iv)}}{\leq} \  \eps/2\;. \]
On the other hand, in the first term in \eqref{eq:mixover} neither chain exits $D$.  Denote by $\mathcal F_n$ the $\sigma$-algebra generated by both copies up to transition step $n$.
Now, by \eqref{minor} and the Markov property, it holds that
\begin{align*}
        &\mathbb P\bigr(X_{k(m+1)}\ne\widetilde X_{k(m+1)},\mathfrak T\geq k(m+1)\bigr) \\
& \quad =\     \mathbb E\Bigr[\mathbb P\Bigr(  \begin{aligned}[t]
                                        &\{X_{k(m+1)}\ne\widetilde X_{k(m+1)}\}\cap\bigcap_{l=0}^m\{X_{(k-1)(m+1)+l},\widetilde X_{(k-1)(m+1)+l}\in D\}\Bigr| \\
                                        &\mathcal F_{(k-1)(m+1)}\Bigr)\,;\,X_{(k-1)(m+1)}\ne\widetilde X_{(k-1)(m+1)},\mathfrak T\geq (k-1)(m+1)\Bigr]
                                        \end{aligned} \\
& \quad =\     \mathbb E\Bigr[ \begin{aligned}[t]
                        &\mathbb P_{(X_{(k-1)(m+1)},\widetilde X_{(k-1)(m+1)})}\Bigr(\{X_{m+1}\ne\widetilde X_{m+1}\}\cap\bigcap_{l=0}^m\{X_l,\widetilde X_l\in D\}\Bigr)\,; \\
                        &X_{(k-1)(m+1)}\ne\widetilde X_{(k-1)(m+1)},\mathfrak T\geq (k-1)(m+1)\Bigr]
                        \end{aligned} \\
& \quad \overset{\eqref{minor}}{\leq}\  e^{-1}\,\mathbb P\Bigr(X_{(k-1)(m+1)}\ne\widetilde X_{(k-1)(m+1)},\mathfrak T\geq (k-1)(m+1)\Bigr) \\
& \quad \ \leq\ \cdots\ \leq\ e^{-k}\,\mathbb P(X_0\ne \widetilde X_0)\ \leq\ e^{-k}\ \leq\ \eps/2\;,
\end{align*}
where we used $\{X_{k(m+1)}\ne\widetilde X_{k(m+1)}\}\subseteq\{X_{(k-1)(m+1)}\ne\widetilde X_{(k-1)(m+1)}\}$ in the first equation, which holds by faithfulness, and the choice of $k$ in the last.
Since the $\TV$-distance to  stationarity $\|\nu\pi^{k(m+1)}-\mu\|_{TV}$ is non-increasing, this shows that $k (m+1)$ transition steps of the adjusted chain suffice for   $\varepsilon$ accuracy.

We are left to show \eqref{minor} by using Assumption~\ref{A_S} {\it(i)} - {\it(iii)}.
Let $x,\tilde x\in D$.
Denote the accept events in the $(n+1)$-th transition, i.e. from $X_n$ to $X_{n+1}$ and $\widetilde X_n$ to $\widetilde X_{n+1}$, by $A_{n+1}$ and $\widetilde A_{n+1}$ respectively.
Let $X^u_n$ and $\widetilde X_n^u$ be the corresponding copies of the underlying unadjusted chain and note that $X_n=X^u_n$ on $\bigcap_{l=0}^{n-1}A_{l+1}$.
Considering just one epoch consisting of $m+1$ transition steps, {\it(iii)} allows to restrict to the case that the Metropolis filter does not intervene over the epoch so that the probability that there exists a coupling of the adjusted chains which induces meeting is determined by the corresponding probability for the underlying unadjusted chains.  More precisely, 
\begin{align*}
        &\mathbb{P}_{(x,\tilde x)}\Bigr(\{X_{m+1}\ne\widetilde{X}_{m+1}\}\cap\bigcap_{l=0}^m\{X_l,\widetilde{X}_l\in D\}\Bigr) \\
& \quad \leq\   \begin{aligned}[t]
        &\mathbb{P}_{(x,\tilde x)}\Bigr(\{X^u_{m+1}\ne\widetilde{X}^u_{m+1}\}\cap\bigcap_{l=0}^m(A_{l+1}\cap\widetilde{A}_{l+1})\cap\bigcap_{l=0}^m\{X_l,\widetilde{X}_l\in D\}\Bigr) \\
        &+ \sum_{l=0}^m\:\Bigr[\mathbb{P}_x\bigr(A_{l+1}^c\cap\{X_l\in D\}\bigr)+\mathbb{P}_{\tilde x}\bigr(\widetilde{A}_{l+1}^c\cap\{\widetilde{X}_l\in D\}\bigr)\Bigr]
        \end{aligned}
\end{align*}
with the second term bounded by $2(m+1)\sup_{x\in D}\mathbb{P}(A(x)^c)\leq2(3e)^{-1}$ by {\it(iii)}.  For the first term, we employ $m$ steps of the contractive coupling in {\it(i)} which brings the two copies of the unadjusted chain sufficiently close together for one step of the regularizing coupling in {\it(ii)} to induce exact meeting.
This yields
\begin{align*}
        &\mathbb{P}_{(x,\tilde x)}\Bigr(\{X^u_{m+1}\ne\widetilde{X}^u_{m+1}\}\cap\bigcap_{l=0}^m(A_{l+1}\cap\widetilde{A}_{l+1})\cap\bigcap_{l=0}^m\{X_l,\widetilde{X}_l\in D\}\Bigr) \\
& \quad \leq\  \mathbb{P}_{(x,\tilde x)}\Bigr(\Pi^u_{Reg}\bigr((X_m^u,\widetilde{X}_m^u),\Delta^c\bigr)\,;\,\bigcap_{l=0}^m\{X^u_l,\widetilde{X}^u_l\in D\}\Bigr) \\
& \quad \overset{{\it(ii)}}{\leq}\    C_{Reg}\:\mathbb{E}_{(x,\tilde x)}\Bigr(d(X_m^u,\widetilde{X}_m^u)\,;\,\bigcap_{l=0}^{m-1}\{X^u_l,\widetilde{X}^u_l\in D\}\Bigr) \\
& \quad \overset{{\it(i)}}{\leq} \    C_{Reg}\,e^{-\rho m}d(x,\tilde x)
\ \leq\    C_{Reg}\,Re^{-\rho m}
\ \leq\    (3e)^{-1}\;,
\end{align*}
where in the last two steps $\diam(D)\leq R$ and the definition of $m$ were used respectively.  
\end{proof}

\begin{remark}[Scope of Coupling Framework]
\label{rmk:broad}
A remarkable feature of the coupling framework  presented in this section is that it uses localization to boost local mixing of the unadjusted kernel to mixing of the adjusted kernel.  This feature is enabled by Assumption~\ref{A_S} {\it(iv)} which localizes the entire coupling argument to the domain $D$.  In particular, the assumption that the unadjusted kernel admits a locally contractive coupling and a local one-shot coupling  (i.e., Assumption~\ref{A_S} {\it(i)} and {\it(ii)})
does not impose global restrictions, such as regularity or convexity, on the stationary distribution.  Therefore, this new coupling framework is broadly applicable including, i.p., to stationary distributions whose log-density is non-globally gradient or Hessian Lipschitz or non-globally concave.
\end{remark}

\section{Application to a non-reversible, adjusted Markov chain}

\label{sec:MAKLA}

Although there are numerous non-asymptotic convergence results for kinetic Langevin diffusions \cite{cheng2018underdamped,cheng2018sharp, dalalyan2020sampling, eberle2019couplings,cao2019explicit} and their unadjusted discretizations \cite{cheng2018underdamped,cheng2018sharp,dalalyan2020sampling,shen2019randomized,monmarche2021high}, quantitative mixing time guarantees for \emph{adjusted} discretizations are comparatively scarce.
In view of this underdevelopment, and as an application of Theorem~\ref{thm:mix}, 
mixing time guarantees for a non-reversible, adjusted Markov chain based on a discretization of the kinetic Langevin diffusion are given in Theorem \ref{thm:MAKLAmix} of this section.

\subsection{Metropolis-adjusted Kinetic Langevin Algorithm (MAKLA)}

Consider an absolutely continuous probability distribution on $\mathbb{R}^d$ of the form $\mu_{target}(dx) \propto e^{-U(x)} dx$, where  $U: \mathbb{R}^d \to \mathbb{R}$ is a continuously differentiable potential energy function.  Here we analyze the mixing of an MCMC method aimed at  $\mu_{target}$ based on the kinetic Langevin diffusion \begin{equation}
\label{eq:uDL}
\mathrm d X_t\ =\ V_t\:\mathrm dt \;, \quad \mathrm d V_t\ =\ -\nabla U(X_t)\:\mathrm dt - \gamma V_t\:\mathrm dt + \sqrt{2 \gamma}\:\mathrm d B_t\;,
\end{equation}      
where $B_t$ is a standard $d$-dimensional Brownian motion and $\gamma>0$ is the friction. 
Let $I_d$ be the $d \times d$ identity matrix.
A key property of \eqref{eq:uDL} is that it leaves invariant the probability measure
\begin{equation}
\label{eq:muk} \mu(dz) \ = \  \mu_{target} \otimes \mathcal{N}(0, I_d)(dx\;dv)  \ \propto \ e^{-H(z)} \; dz
\end{equation}
%\[ \mu(dx \; dv) \ = \  \mu_{target} \otimes \mathcal{N}(0, A^{-1})( dx \; dv)  \ \propto \ e^{-H(x,v)} \; dx \; dv \]
on phase space $z=(x,v)\in\mathbb R^{2d}$ with energy
\[ H(z)\ =\ \frac{1}{2} |v|^2\ +\ U(x)\;. \]

%\[ H(x,v)\ =\ \frac{1}{2} \|v\|_A^2 + \frac{1}{2}   \|x\|_A^2 + U(x)\ =\ \frac{1}{2}\|z\|+U(x)\;. \]

A variety of discretizations of \eqref{eq:uDL} can  be Metropolis-adjusted \cite{ScLeStCaCa2006,LeRoSt2010, BoVa2010,Bo2014} and fit the framework \eqref{eq:X}. Here we focus on a symmetric Strang splitting   \cite{BuDoPa2007,Bo2014,BePiSaSt2011},  where the splitting components are given by
\begin{enumerate}[label=\arabic*.]
\item  the Ornstein-Uhlenbeck (OU) flow
\[ \mathrm d X_t\ =\ 0 \;, \quad \mathrm d V_t\ =\ - \gamma V_t\:\mathrm dt + \sqrt{2 \gamma}\:\mathrm d B_t\;, \]
\item  the purely potential flow
\[ \mathrm d X_t\ =\ 0 \;, \quad \mathrm d V_t\ =\ -\nabla U(X_t)\:\mathrm dt\;, \quad \text{and} \]
\item the purely kinetic flow
\[ \mathrm d X_t\ =\ V_t\:\mathrm dt \;, \quad \mathrm d V_t\ =\ 0\;. \]
\end{enumerate}

\begin{comment}
\begin{equation*}
\begin{aligned}
 \begin{pmatrix}
 dX_t = 0  \\
 dV_t =  - \gamma V_t dt + \sqrt{2 \gamma} d B_t
 \end{pmatrix} \;, ~
 \begin{pmatrix}
d X_t = 0  \\
 d V_t =  -  \nabla U(X_t) dt
 \end{pmatrix} \;,  ~
\begin{pmatrix}
 d X_t = V_t dt  \\
d V_t = 0  
 \end{pmatrix} \;.
 \end{aligned} 
\end{equation*}
\end{comment}

\noindent
The corresponding discretized flows are for
\begin{enumerate}[label=\arabic*.]
\item  the OU-substep 
 \begin{align}
 \label{eq:OU}
    O_{h}(\mathsf{b}) (x,v) \ &= \ (x, e^{-\gamma h } v + (1- e^{-2 \gamma h } )^{\frac{1}{2}}  \, \mathsf{b}  ) \;,   \quad \mathsf{b} \in \mathbb{R}^d \;,
    \end{align}

\item  the B-substep for the \emph{kick} due to the potential part 
    \begin{align}
\label{eq:B_flow} \theta_h^{(B)}(x,v) \ = \
    \Big( x,~ v - h \nabla U(x) \Big) \;, \quad \text{and}
    \end{align}
    \item the A-substep for the \emph{drift} due to the kinetic part
\begin{equation}
\label{eq:A_flow} \theta_h^{(A)}(x,v)\ =\ 
    \Big( x+ h  v, v \Big) \;.
\end{equation}
\end{enumerate}

Combining these flow maps in the following palindromic fashion yields the \emph{unadjusted kinetic Langevin algorithm} (UKLA) with transition step given by  \begin{equation}
(X_1^u, V_1^u) \ = \ O_{h/2}(\xi_{2}) \circ \theta^{(A)}_{h/2} \circ \theta^{(B)}_h \circ \theta^{(A)}_{h/2} \circ O_{h/2}(\xi_{1}) (X_0^u, V_0^u)\;,
\label{eq:OABAO} 
\end{equation} where  $\xi_{1}, \xi_{2}$ are i.i.d.~random variables with distribution $\mathcal{N}(0,I_d)$.  This discretization  is commonly referred to as ``OABAO'' where each letter refers to either \eqref{eq:OU}, \eqref{eq:B_flow} or \eqref{eq:A_flow}.
For the sequel, it is convenient to introduce
\begin{equation}
\label{eq:thetah}
\theta_h\ =\ \theta_{h/2}^{(A)} \circ \theta_{h}^{(B)} \circ \theta_{h/2}^{(A)}\;. \end{equation}
By construction, $\theta_h$ is both volume-preserving and reversible \cite{HaLuWa2010,BoSaActaN2018}.
The transition kernel  of UKLA is given by $\pi^u = \Xi \Theta \Xi$, where
\begin{align*}
\Xi((x,v), \,\cdot\, ) \ &= \ \delta_x \otimes \mathcal{N}\big(e^{-\gamma h/2} v, (1-e^{-\gamma h }) \,  I_d \big) \;, \\
\Theta((x,v), \,\cdot\,) \ &= \  \delta_{\theta_{h}(x,v)}\;.
\end{align*}
Due to asymptotic bias,   UKLA does not leave $\mu$ invariant, i.e., $\mu\pi^u\ne\mu$.   This failure is not surprising, since although the OU steps leave $\mu$ invariant and $\theta_h$ is volume-preserving, the time discretization induces an energy error under $\theta_h$, i.e., $(H\circ\theta_h-H) \not\equiv  0$, which is the root cause of the asymptotic bias.

\medskip

The OABAO scheme can be readily Metropolis-adjusted by simply adjusting $\theta_h$, which is possible since $\theta_h$ is both volume-preserving and reversible \cite[Prop.~5.1]{BoSaActaN2018}; see also \cite[Theorem 2]{Tierney1998}.
The resulting algorithm is called the \emph{Metropolis-adjusted kinetic Langevin algorithm}  (MAKLA)  with transition step \begin{equation}
(X_1, V_1) \ = \ O_{h/2}(\xi_{2}) \circ \hat{\theta}_h(\mathcal{U}) \circ O_{h/2}(\xi_{1}) (X_0, V_0)\:,  \label{eq:OMABAO} 
\end{equation} where $\mathcal{U}\sim\mathrm{Unif}(0,1)$ is independent of the other random variables and the state of the chain, and the Metropolis-adjusted  integrator is defined through the mapping \begin{align}
\hat{\theta}_h(\mathsf{u})(x,v) \ = \ \begin{cases} \theta_h(x,v) & \text{if $\mathsf{u} \leq \alpha((x,v),\theta_h(x,v))$,} \\
\mathcal{S}(x,v) & \text{else,}
\end{cases}
\label{eq:thetahat}
\end{align}
where $\alpha((x,v),(x',v'))=\exp(-(H(x',v')-H(x,v))^+)$ is the accept probability, $\mathcal{S}(x,v)=(x,-v)$ is the velocity flip involution, and $[ \cdot ]^+ = \max(0, \cdot)$.
The transition kernel  of MAKLA is  $\pi = \Xi \widetilde\Theta \Xi$ with
\[
\widetilde\Theta((x,v), dx' \, dv') \ = \
    \begin{aligned}[t]
    &\alpha((x,v), (x',v')) \, \delta_{ \theta_{h}(x,v) } ( dx' \, dv') \\
    &+ \big(1 - \alpha((x,v),(x',v')) \big) \, \delta_{\mathcal{S}(x,v)}(dx' \, dv')  \;.
    \end{aligned}
\]
It is easily verified that $\pi$  leaves $\mu$ invariant, i.e., $\mu\pi = \mu$, and therefore, the $x$-marginal of the corresponding Markov chain can be used to  sample from $\mu_{target}$.

\subsection{Assumptions \& Additional Notation}

  For simplicity, we focus on strongly log-concave target distributions with gradient Lipschitz log-densities having bounded third derivatives.
More precisely, we fix the following assumptions:

\begin{assumption}{\bf1}\label{A_K}
Suppose $U$ is $K$-strongly convex, i.e., there exists $K>0$ such that \[
\bigr( \nabla U(x) - \nabla U(y)\bigr) \cdot (x - y) \ \ge \ K | x - y |^2 \quad \text{for all $x,y \in \mathbb{R}^d$.}
\]
\end{assumption}

\begin{assumption}{\bf2}\label{A_L}
Suppose $U$ has a global minimum at $0$, $U(0)=0$, and $U$ is $L$-gradient Lipschitz continuous, i.e., there exists $L >0$ such that \[
\bigr|\nabla U(x) - \nabla U(y)\bigr| \ \le \ L \, |x-y| \quad \text{for all $x,y \in \mathbb{R}^d$.} \]
\end{assumption}

Below it is sometimes convenient to write the results and conditions in terms of the \emph{condition number} of the target distribution defined in the usual way by $\kappa = L/K$.  Define the third derivative via the trilinear product
\[ \nabla^3U(x)\,\vdots\,(a\otimes b\otimes c)\ =\ \sum\nolimits_{i,j,k=1}^d\partial^3_{ijk}U(x)\,a_ib_jc_k\quad \text{for $x,a,b,c \in \mathbb{R}^d$.} \]

\begin{assumption}{\bf3}\label{A_LH}
Suppose $U\in C^3(\mathbb{R}^d)$ is $L_H$-Hessian Lipschitz, i.e., there exists $L_H \ge 0$ such that for all $x,y\in\mathbb R^d$, it holds
\[ \bigr|\nabla^3U(x)\,\vdots\,(a\otimes b\otimes c)\bigr|\ \leq\ L_H|a|\,|b|\,|c|\quad \text{for all $x,a,b,c \in \mathbb{R}^d$.} \]
\end{assumption}

Define the sets of model parameters and user-specified hyperparameters to be $\mathcal M=\{d,K,L,L_H\}$ and $\mathcal H=\{\eps,\nu,\gamma,h\}$, respectively.
Since we mainly care about the non-logarithmic dependencies of the mixing time on the underlying model parameters, and for the sake of legibility of expressions, we often suppress logarithmic dependencies on parameters in $\mathcal M$ by using the notation: for two quantities $\mathsf{x},\mathsf{y}\in\mathbb R$, we write $\mathsf{x}=\widetilde{\mathcal{O}}(\mathsf{y})$ if there exists $C>0$ depending at most logarithmically on any parameter in $\mathcal M$  such that $\mathsf{x}\leq C \mathsf{y}$. The symbol $\mathcal O$ is defined similarly except that it expresses all logarithmic dependencies.

\begin{assumption}{\bf4}
Regarding the user-tuned hyperparameters, let $0<\eps\leq1/2$ and suppose $\nu\in\mathcal P(\mathbb R^d)$ such that $\log\nu(e^{H/8})$ depends at most polynomially on the model parameters, i.e., there exist constants $n_1,n_2,n_3,n_4\in\mathbb Z$ such that
\begin{equation}\label{nuO}
    \log\nu(e^{H/8})\ =\ \widetilde{\mathcal{O}}\bigr(d^{n_1}K^{n_2}L^{n_3}L_H^{n_4}\bigr)\;.
\end{equation}
Further, let $\gamma,h>0$ satisfy
\[ L^{1/2}\gamma^{-1}\ \leq\ 1/10\quad\text{and}\quad\gamma h\ \leq\ 1\;, \]
as well as $\log(1/h)=\widetilde{\mathcal{O}}(1)$.
\label{A_h}
\end{assumption}
Note that \eqref{nuO} and the last part of \ref{A_h} pose no relevant restriction because exponential dependencies on model parameters of the quantities of interest are unrealistic.

\begin{remark}[Possibilities to Relax the Assumptions]
There are several possibilities the assumptions made above can be relaxed while sustaining the mixing guarantees of Theorem~\ref{thm:MAKLAmix}.
First, the global strong convexity assumption in \ref{A_K} can be relaxed to asymptotic strong convexity by employing a more sophisticated coupling in \ref{A_S} {\it(i)} as developed in \cite{BoEbZi2020,cheng2018sharp, BouRabeeSchuh2023}. However, the resulting contraction rates will depend on underlying parameters in a more intricate way.
Second, as emphasized in Remark~\ref{rmk:broad}, both the global gradient and Hessian Lipschitz continuity in \ref{A_L} and \ref{A_LH} as well as the global convexity in \ref{A_K} can be replaced with local versions.  In particular, convexity in a suitable shell suffices.
\end{remark}

\subsection{Mixing Guarantees for MAKLA}

We are now in position to state upper bounds on the mixing time of MAKLA as defined in \eqref{eq:tmix} with $\mu$ given by \eqref{eq:muk}.
\begin{theorem}\label{thm:MAKLAmix}
Suppose Assumptions~\ref{A_K}-\ref{A_h} hold. Then there exists $\bar{h}>0$ with
\[
        (\overline{h})^{-1}
\ =\    \widetilde{\mathcal{O}}\Bigr[
        \begin{aligned}[t]
        &(L^{1/2}\gamma^{-1})^{-1/2}\kappa\log(1/\eps) \\
        &\times\Bigr(
            \begin{aligned}[t]
            &L_H^{1/2}K^{-1/4}d^{3/4}\max\bigr((L^{1/2}\gamma^{-1})^{-2},(\kappa d)^{-1}\log\nu(e^{H/8})\bigr)^{3/4} \\
            &+L^{1/2}d^{1/2}\max\bigr((L^{1/2}\gamma^{-1})^{-2},(\kappa d)^{-1}\log\nu(e^{H/8})\bigr)^{1/2}\Bigr)\Bigr]
            \end{aligned}
        \end{aligned}
\]
such that for all $h\leq\overline{h}$, it holds that
\[ \tmix(\eps,\nu)\ =\ \widetilde{\mathcal{O}}\bigr(h^{-1}K^{-1}\gamma\log(1/\eps)\bigr) \;. \]

\end{theorem}

For a fixed step size $h\leq\overline{h}$, Theorem~\ref{thm:MAKLAmix} guarantees that starting in $\nu$, $\tmix(\eps,\nu)$ transition steps of MAKLA suffice to ensure $\eps$-accuracy  in $\TV$.
The assumptions on the initial distribution are minimal.
In particular, cold start distributions are covered, i.e.,  $\nu=\delta_z$ for some $z\in\mathbb R^{2d}$.

\begin{remark}[Mixing Guarantee]
Note that if
\begin{equation*}
\gamma\ =\ \mathcal O(L^{1/2})\qquad\text{and}\qquad\log\nu(e^{H/8})\ = \ \widetilde{\mathcal{O}}(\kappa d)\;,
\end{equation*}
Theorem \ref{thm:MAKLAmix} asserts that for $h=\overline{h}$,
\begin{equation}\label{mixresult}
    \tmix(\eps,\nu)\ =\ \widetilde{\mathcal{O}}\Bigr[\kappa^{3/2}\max\bigr(L_H^{1/2}(d/K)^{3/4},\,L^{1/2} (d/K)^{1/2}\bigr)\log^2(1/\eps)\Bigr]
\end{equation}
since in this case
\begin{equation}\label{hresult}
        (\overline{h})^{-1}
\ =\    \widetilde{\mathcal{O}}\Bigr[\kappa^{1/2} L^{1/2}\max\bigr(L_H^{1/2}(d/K)^{3/4},\,L^{1/2}( d/K)^{1/2}\bigr)\log(1/\eps)\Bigr]\;.
\end{equation}
This choice of $\gamma$ minimizes the mixing time upper bound while still satisfying \ref{A_h}.  Moreover, the assumption on $\nu$ is mild; e.g.,  it is satisfied by all cold starts in $z\in\mathbb R^{2d}$ such that $H(z)/8=\log\delta_z(e^{H/8})=\mathcal{O}(\kappa d)$.
To put this in perspective, note that the Gaussian measure $\nu=\mathcal N(0,A^{-1})\otimes\mathcal N(0,I_d)$ with energy $H(z)=\frac{1}{2}|v|^2+\frac{1}{2}|A^{1/2}x|^2$ amounts to $\log\nu(e^{H/8})=d\log(8/7)$.
\end{remark}

\begin{remark}[Dimension Dependence]
Remarkably, the dimension scaling obtained in \eqref{mixresult} is optimal in the \emph{high acceptance regime}, cf. Definition~\ref{def:HAR}, from a cold start distribution as illustrated by
\begin{equation}\label{exU}
    U(x)\ =\ \frac{1}{2}x\cdot\operatorname{diag}\bigr(2,1,\dots\bigr)x-\sin(x_1)\;.
\end{equation}
%This potential is $1$-strongly convex, $3$-gradient Lipschitz, and its third derivative is bounded by $1$.
Denote by $e_1$ the unit vector in the first component and consider the collection $\mathcal C=\{\delta_{(0,d^{1/2}e_1)}\pi^n:n\geq0\}$ corresponding to a cold start in $(0,d^{1/2}e_1) \in \mathbb{R}^{2d}$.
According to \eqref{HAR}, $\pi$ being in the high acceptance regime on $\mathcal C$ requires
\begin{equation}\label{HARex}
    h^{-1}\,\mathbb P\bigr(A(0,d^{1/2}e_1)^c\bigr)\ =\ \mathcal O(1)\;,
\end{equation}
where we used that the mixing time $\tmix^{exact}$, cf. \eqref{tmixexact}, of the transition kernel of the kinetic Langevin diffusion over time $h$ is of order $h^{-1}$.
Expanding \eqref{eq:deltaH} shows that the energy error to leading order is
\[ (H\circ\theta_h-H)(x,v)\ =\ \frac{h^3}{24}\bigr(\nabla^3U(x)\,\vdots\,v^{\otimes3}-6v\cdot\nabla^2U(x)\nabla U(x)\bigr)+\mathcal O(h^4)\;. \]
For $U$ as in \eqref{exU}, it hence holds that
%\[ (H\circ\theta_h-H)\circ O_{h/2}(\xi_1)(0,d^{1/2}e_1)\ =\ \frac{h^3}{24}\bigr(d^{3/2}+12d^{1/2}\bigr)+\mathcal O(h^{7/2})\;. \]
\[ (H\circ\theta_h-H)(0,d^{1/2}e_1)\ =\ \frac{h^3}{24}\bigr(d^{3/2}+12d^{1/2}\bigr)+\mathcal O(h^4)\;. \]
Since the reject probability from cold start in $(0,d^{1/2}e_1)$ is given by
\[ \mathbb P\bigr(A(0,d^{1/2}e_1)^c\bigr)\ =\ 1-\mathbb Ee^{-(H\circ\theta_h-H)\circ O_{h/2}(\xi_1)(0,d^{1/2}e_1)^+}\;, \]
and the OU step to leading order in $h$ equals the identity, \eqref{HARex} implies
\[ h^2d^{3/2}\ =\ \mathcal O(1)\;. \]
%However, it is worth mentioning that MAKLA possesses an intriguing self-healing property that in principle allows the algorithm to escape bottlenecks caused by Metropolis-adjustment.
%More precisely, MAKLA flips the velocity in case of rejection so that the lowest order term in the energy error, which is odd in the velocity, flips sign.
%Rejections caused by this term being large hence do not appear twice in a row.
\end{remark}

\begin{remark}[Condition Number Dependence]
In Lemma \ref{lem:OABAO_contr}, UKLA is shown to converge to its stationary distribution with rate $\rho \propto K\gamma^{-1}h$, which under \ref{A_h} is at best $\kappa^{-1}$ for $h^{-1}$ of order $L^{1/2}$.
This rate differs from the optimal rate obtained for the kinetic Langevin diffusion under warm start \cite{cao2019explicit}.
Passing to MAKLA via Theorem \ref{thm:mix} further increases scaling in condition number.
In \eqref{mixresult}, the additional $\kappa^{1/2}$ in front of the maximum is expected for \ref{A_S} {\it(iii)} to hold with $\rho$ and the energy error bounds of Lemma \ref{lem:posVerlet_deltaH}.
However, due to the linear appearance of $\gamma h \mathfrak H$ in \eqref{RU} (cf. Remark \ref{rmk:flip}), \ref{A_S} {\it(iii)} requires the extra $K^{-3/4}$ and $K^{-1/2}$ inside the maximum.
At present, the optimal condition number dependence for either UKLA or MAKLA from a cold start distribution is not known.
\end{remark}

%Choosing the step size such that \ref{A_S} {\it(iii)} holds, yields \eqref{hresult} and therefore a different scaling for MAKLA.
%Due to asymptotic bias, however, UKLA requires a much more costly $\eps$ scaling while our result captures the logarithmic dependence expected for unbiased algorithms such as MAKLA.

\begin{proof}[Proof of Theorem~\ref{thm:MAKLAmix}]
To invoke Theorem \ref{thm:mix}, it suffices to verify Assumption \ref{A_S} for MAKLA, which as described below, relies on  ingredients developed in \S\ref{sec:MAKLAmix_ingredients}.

\smallskip

Let $S=\mathbb R^{2d}$ with metric induced by the twisted norm $\twnorm{\cdot}$ defined in \eqref{eq:tw}. Define the domain $D=\{\mathcal E(z)\leq R_U\}$ for some $R_U\geq2$ to be determined momentarily and where $\mathcal{E}$ is the energy-like function defined in \eqref{eq:calE}.
Note that $K|x|\leq|\nabla U(x)|$ by \ref{A_K} and \ref{A_L}, and hence,
\[ \twnorm{z}^2\ \overset{\eqref{eq:tw_equiv}}{\leq} \ \frac{17}{16}\bigr(|v|^2+\gamma^2|x|^2\bigr)\ \leq\ \frac{17}{16}\frac{L\gamma^2}{K^2}\mathcal E(z)\;, \]
where in the last step we used \ref{A_h} to factor out $L\gamma^2/K^2\geq36\kappa^2\geq1$.
Thus,
\[ \diam_{\twnorm{\cdot}}(D)\ \leq\ 2\sup_{z\in D}\twnorm{z}\ \leq\ 3\frac{L^{1/2}\gamma}{K}R_U^{1/2} \ =: \ R \;, \]
which specifies $R$ in \ref{A_S}.

\smallskip

\noindent
Regarding the unadjusted transition kernel, 
\begin{itemize}
\item Assumption~\ref{A_S} {\it(i)}
holds by Lemma \ref{lem:OABAO_contr} and \ref{A_h} with rate
\[ \rho\ =\ \frac{\gamma h}{34\sqrt{e}}\min\bigr(K\gamma^{-2},1\bigr)\ =\ \frac{K\gamma^{-1}h}{34\sqrt{e}}\;; \quad \text{and,}  \]
\item Assumption~\ref{A_S} {\it(ii)}
holds by Lemma \ref{lem:overlap_pukLa} and \eqref{eq:tw_equiv} with
\[ C_{Reg}\ =\ 14\bigr((\gamma h)^{-3/2}+\gamma^{-1}L_Hd^{1/2}h^2\bigr)\;. \]
\end{itemize}
This completes the verification of Assumption~\ref{A_S} {\it(i)} and {\it(ii)}.

It remains to verify Assumption~\ref{A_S} {\it(iii)} and {\it(iv)}, which concern the adjusted transition kernel.  To this end, the epoch of transition steps $\mathfrak E$ and the total number of transition steps $\mathfrak H$ play a pivotal role.
Assumption \ref{A_h} implies $\log(\gamma C_{Reg})=\widetilde{\mathcal{O}}(1)$.  Thus,  
\[ \mathfrak E\ =\ \widetilde{\mathcal{O}}\bigr(K^{-1}\gamma h^{-1} \log R_U\bigr)\quad\text{and}\quad\mathfrak H\ =\ \widetilde{\mathcal{O}}\bigr(K^{-1}\gamma h^{-1} \log R_U\log(1/\eps)\bigr)\;. \]
Since {\it(iii)} depends on $R_U$, which needs to be chosen sufficiently large for the exit probability bound in {\it(iv)} to hold, we first verify {\it(iv)}.

%and since $K \le L$, note that $h=\mathcal{O}(K^{-1}\gamma)$.

\begin{itemize}
\item To verify Assumption~\ref{A_S}  {\it(iv)}, we invoke Lemma \ref{lem:exit} as follows.
By Lemma \ref{lem:posVerlet_deltaH}, $C_{\Delta H}=4L$ and $k=2$ in Lemma \ref{lem:exit}.
Moreover, \eqref{exit_h} holds due to \ref{A_h}.
Define $\overline{h}_1>0$ to saturate the bound $400L\mathfrak Hh^2\leq1$ and let $h\leq\overline{h}_1$.
%Let $h\leq\overline{h}_1$, where $\overline{h}_1$ solves $144L\mathfrak H(\overline{h}_1)^2=1$.
We now select $R_U$ to counter-saturate the bound
\begin{equation}\label{RU}
    R_U\ \geq\ 32\Bigr[\gamma h \mathfrak Hd+\log\Bigr(\frac{4}{\eps}\max\bigr(\nu(e^{H/8}),(2\kappa)^{d/2}\bigr)\Bigr)\Bigr]\;,
\end{equation}
where we inserted $\mu(e^{H/8})\leq(2\kappa)^{d/2}$ by \ref{A_K} and \ref{A_L}.  By Lemma \ref{lem:exit}, this choice of $R_U$ ensures  Assumption~\ref{A_S} {\it(iv)} to hold starting from both $\nu$ and $\mu$.
Since the right hand side of \eqref{RU} depends logarithmically on $R_U$, note that
\[ R_U\ =\ \widetilde{\mathcal{O}}\Bigr(\max\bigr(K^{-1}\gamma^2d,\,\log\nu(e^{H/8})\bigr)\log(1/\eps)\Bigr)\;, \]
which implies
$(\overline{h}_1)^{-1}\ =\ \widetilde{\mathcal{O}}\bigr(\kappa\gamma\log R_U\log(1/\eps)\bigr)\ =\ \widetilde{\mathcal{O}}\bigr(\kappa\gamma\log(1/\eps)\bigr) $.
\item Finally, we verify Assumption~\ref{A_S}  {\it(iii)}.
Leveraging: (a) the higher order bound of Lemma \ref{lem:posVerlet_deltaH}; (b) the bounds 
\begin{align*}
    &\mathbb E\,\mathcal E(O_{h/2}(\xi_1)(z))\ 
    \leq\ \mathcal E(z)+\gamma hd\;,\quad\text{as well as}\\
    &\mathbb E\,\mathcal E(O_{h/2}(\xi_1)(z))^{3/2}\ \leq\ 4\bigr(\mathcal E(z)^{3/2}+3(\gamma hd)^{3/2}\bigr)
\end{align*}
that each hold for all $z\in\mathbb R^{2d}$; and (c) the definition of $D$ yields
\begin{align*}
        & \mathfrak E\:\sup_{z\in D}\mathbb P(A(z)^c)
\ \leq\ \mathfrak E\:\sup_{z\in D}\mathbb E|\Delta H\circ O_{h/2}(\xi_1)(z)| \\
&\quad \leq\ \mathfrak Eh^3\sup_{z\in D}\bigr(2L_H\mathbb E\,\mathcal E(O_{h/2}(\xi_1)(z))^{3/2}+L^{3/2}\mathbb E\,\mathcal E(O_{h/2}(\xi_1)(z))\bigr) \\
&\quad \leq\ \mathfrak Eh^3\Bigr(8L_H\bigr(R_U^{3/2}+3(\gamma hd)^{3/2}\bigr)+L^{3/2}\bigr(R_U+\gamma hd\bigr)\Bigr)\;.
\end{align*}
Inserting $\mathfrak E$ and $R_U$, and using that $\gamma h d=\mathcal{O}(K^{-1}\gamma^2d)$ which holds by \ref{A_h} and $K \le L$, shows that there exists $\overline{h}_2>0$ such that the last display is bounded by $(3e)^{-1}$ for all $h\leq\overline{h}_2$ with
\[
        (\overline{h}_2)^{-1}
\ =\    \widetilde{\mathcal{O}}\Bigr[
        \begin{aligned}[t]
        &L_H^{1/2}
            \begin{aligned}[t]
            &K^{-1/4}\kappa d^{3/4}\log^{3/4}(1/\eps)(L^{1/2}\gamma^{-1})^{-1/2} \\
            &\times\max\bigr((L^{1/2}\gamma^{-1})^{-2},(\kappa d)^{-1}\log\nu(e^{H/8})\bigr)^{3/4}
            \end{aligned} \\
        &+L^{1/2}
            \begin{aligned}[t]
            &\kappa d^{1/2}\log^{1/2}(1/\eps)(L^{1/2}\gamma^{-1})^{-1/2} \\
            &\times\max\bigr((L^{1/2}\gamma^{-1})^{-2},(\kappa d)^{-1}\log\nu(e^{H/8})\bigr)^{1/2}\Bigr]\;.
            \end{aligned}
        \end{aligned}
\]
\end{itemize}
This completes the verification of Assumption~\ref{A_S} {\it(iii)} and {\it(iv)}.  To finish, set $\overline{h}=\min(\overline{h}_1,\overline{h}_2)$ which satisfies, by \ref{A_h},
\[
        (\overline{h})^{-1}
\ =\    \widetilde{\mathcal{O}}\Bigr[
        \begin{aligned}[t]
        &(L^{1/2}\gamma^{-1})^{-1/2}\kappa\log(1/\eps) \\
        &\times\Bigr(
            \begin{aligned}[t]
            &L_H^{1/2}K^{-1/4}d^{3/4}\max\bigr((L^{1/2}\gamma^{-1})^{-2},(\kappa d)^{-1}\log\nu(e^{H/8})\bigr)^{3/4} \\
            &+L^{1/2}d^{1/2}\max\bigr((L^{1/2}\gamma^{-1})^{-2},(\kappa d)^{-1}\log\nu(e^{H/8})\bigr)^{1/2}\Bigr)\Bigr]\;.
            \end{aligned}
        \end{aligned}
\]  Since Assumption \ref{A_S} holds,   Theorem~\ref{thm:MAKLAmix} now follows by invoking Theorem~\ref{thm:mix}.  
\end{proof}

\section{Key Ingredients of the Proof of Theorem~\ref{thm:MAKLAmix}}

\label{sec:MAKLAmix_ingredients}

The ingredients needed in the proof of Theorem~\ref{thm:MAKLAmix} are developed in this section.  

\subsection{Verifying \ref{A_S} {\it(i)}: Contractive Coupling for UKLA}

Here we use a coupling $\Pi^u_{Contr}$ of two copies of UKLA starting from different initial distributions to demonstrate $L^1$-Wasserstein contractivity with respect to a \emph{twisted metric} induced by the \emph{twisted norm} on $\mathbb R^{2d}$
\begin{equation} \label{eq:tw}
\begin{aligned}
 \twnorm{ (x,v) }^2 := \alpha \ |  x|^2 +   \beta \langle  x, v \rangle + | v|^2   ~~   \text{where} ~~ \alpha \,= \, \frac{\beta}{h} \sinh(\frac{\gamma h}{2}) ~~\text{and} ~~ \beta \,=\, \frac{\gamma}{4} \;. 
    \end{aligned}
\end{equation} By using the elementary inequality \begin{equation}  \label{ieq:sinh}
\mathsf{x} \ \le \ \sinh(\mathsf{x}) \ \le \ \frac{6}{5} \mathsf{x} \quad \text{valid for all $\mathsf{x} \in [0,1]$}
%\footnote{The lower bound is standard, and the upper bound comes from the hyperbolic Cusa-Huygen's inequality \cite{neuman2010some}, i.e., $\sinh(\mathsf{x}) \ \le \ (2 + \cosh(\mathsf{x})) /3$ and using $\cosh(1) < 8/5$.}
,
\end{equation} note that \begin{equation} 
\frac{1}{8} \gamma^2
\ \le \ \alpha \ \le  \ \frac{3}{20} \gamma^2 \;. \label{ieq:alpha}
\end{equation}
Hence, the twisted norm compares to the \emph{untwisted norm} 
\begin{equation} \label{eq:utw}
    \| (x,v) \|^2 \ := \ \gamma^2 | x|^2 + | v|^2
\end{equation}
via
\begin{equation} \label{eq:tw_equiv}
\frac{1}{16} \| (x,v) \|^2 \ \le  \ \| (x,v) \|_{\tw}^2 \ \le \  \frac{17}{16} \| (x,v) \|^2 \;.
\end{equation}

\begin{lemma}\label{lem:OABAO_contr}
Suppose that Assumptions~\ref{A_K},~\ref{A_L}, and~\ref{A_h} hold.
Then, for all $z=(x,v),\tilde{z}= ( \tilde{x}, \tilde{v} ) \in \mathbb{R}^{2d}$ and $\mathsf{a}_1,\mathsf{a}_2 \in \mathbb{R}^d$, it holds that  
\begin{align*}
&    \twnorm{ O_{h/2}(\mathsf{a}_2) \circ \theta_h \circ O_{h/2}(\mathsf{a}_1) (z) - O_{h/2}(\mathsf{a}_2) \circ \theta_h \circ O_{h/2}(\mathsf{a}_1) ( \tilde{z} )}
\\
& \quad \ \le \ (1-  \, c \, h) \ 
\twnorm{ z - \tilde{z} } \quad \text{where} \quad c \, = \, \frac{1}{34 e^{\frac{1}{2}}} \min\left( K \gamma^{-1}, \gamma \right)   \;. 
\end{align*}
\end{lemma}

Under similar assumptions, variants of Lemma~\ref{lem:OABAO_contr} have beeen proven elsewhere in the literature; see, e.g.,  \cite{monmarche2021high, sanz2021wasserstein, leimkuhler2023contraction}.  For the convenience of the reader, however, a complete proof is given below.  As emphasized in previous works, a key ingredient in the proof  of Lemma~\ref{lem:OABAO_contr}  is the \emph{co-coercivity} property of $\nabla U$, which in terms of the potential force $F$ can be written as \begin{equation} \label{eq:cocoercive}
|F(x_1) - F(x_2)|^2 \le - L \, \langle F(x_1) - F(x_2),  x_1-x_2 \rangle \quad \text{for all $x_1, x_2 \in \mathbb{R}^d$} \;.   
\end{equation}
This holds if $U$ is continuously differentiable, convex, and $L$-gradient Lipschitz \cite[Theorem 2.1.10]{nesterov2018lectures}.  
Additionally, the following elementary inequality is used 
\begin{align}
%1 - e^{-\mathsf{x}} &\le \mathsf{x} \qquad\qquad \text{valid}~~\forall~~ \mathsf{x} \ge 0 \label{ieq:exp_1} \;, \\
e^{-\mathsf{x}} & \le 1 - \frac{\mathsf{x}}{2} \qquad \text{valid for all $\mathsf{x} \in [0,1] $.}  \label{ieq::exp} 
\end{align}

\begin{proof}
It is notationally convenient to define  \begin{align*} 
& 
Z \, := \,  O_{h/2}(\mathsf{a}_2) \circ \tilde{\theta}_h \circ O_{h/2}(\mathsf{a}_1) (z) \;, \quad 
\tilde{Z} := O_{h/2}(\mathsf{a}_2) \circ \tilde{\theta}_h \circ O_{h/2}(\mathsf{a}_1) ( \tilde{z}) \,,  \\
& \Delta \, := \,  \| Z - \tilde{Z} \|_{\tw}^2 - \| z - \tilde{z} \|_{\tw}^2 \;, \quad  \zeta \, := \, x - \tilde{x} \;, \quad \omega \, := \, v - \tilde{v} \;, \\
& v_O \, := \, e^{-\frac{\gamma h}{2}} \ v + \sqrt{1-e^{-\gamma h}} \mathsf{a}_1 \;, \quad \tilde{v}_O \, := \,e^{-\frac{\gamma h}{2}} \ \tilde{v} +  \sqrt{1-e^{-\gamma h}} \mathsf{a}_1 \;, \quad   \hat{\omega} \, := \, e^{-\frac{\gamma h}{2}} \ \omega   \;, \quad \\
& \zeta^{\star} \, := \,  \zeta + \frac{h}{2} \hat{\omega} \;, ~~  \Delta F^{\star} \,  := \, F(x+ \frac{h}{2} v_O) - F(\tilde{x}+\frac{h}{2} \tilde{v}_O) \;,  ~~\text{and} ~~ \Phi^{\star} := - \langle \Delta F^{\star} , \zeta^{\star} \rangle   \;.   
\end{align*}  
By definition of the OABAO scheme in \eqref{eq:OABAO}, \[ Z- \tilde{Z}\ =\ \big(  \zeta + h \hat{\omega} + \frac{h^2}{2} \Delta F^{\star} , ~e^{-\frac{\gamma  h}{2}  } ( \hat{\omega} + h \Delta F^{\star} ) \big) \;. \] Inserting this difference into $\Delta$ yields \begin{equation} \label{eq:delta}
\begin{aligned}
 \Delta  \ =  \     \begin{aligned}[t]
                    &\alpha \left( - h^2 \Phi^{\star} + h^2 |\hat{\omega}|^2 + \frac{h^3}{2} \langle \hat{\omega}, \Delta F^{\star} \rangle + \frac{h^4}{4} |\Delta F^{\star} |^2 \right)   \\ 
                    &+ \beta e^{-\frac{\gamma h}{2}} \left(- h \Phi^{\star} + h  |\hat{\omega}|^2  + h^2  \langle \hat{\omega}, \Delta F^{\star} \rangle + \frac{h^3}{2}  |\Delta F^{\star}|^2 \right) \\
                    &+ e^{- \gamma h } \left( \vphantom{ \frac{h^3}{2}} (e^{- \gamma h } - e^{\gamma h} ) |\omega|^2 + 2 h  \langle \hat{\omega}, \Delta F^{\star} \rangle + h^2  |\Delta F^{\star}|^2 \right) \;.
                    \end{aligned}
\end{aligned}  
\end{equation}
Note that the cross-terms involving $\langle \zeta, \omega \rangle$ vanish by definition of $\alpha$ in \eqref{eq:tw}.  Applying Young's inequality with parameters $\delta_1,\delta_2,\delta_3>0$  yields \begin{equation} \label{ieq:Delta_1}
\begin{aligned}
 \Delta \ \le \     \begin{aligned}[t]
                    &\alpha \left( - h^2 \Phi^{\star} + \left( h^2 + \frac{h^3}{4  \delta_1}  \right)  |\hat{\omega}|^2 +  \left( \frac{h^3 \delta_1}{4}  + \frac{h^4}{4} \right) |\Delta F^{\star} |^2 \right) \\ 
                    &+ \beta e^{-\frac{\gamma h}{2}} \left(- h \Phi^{\star} + \left(  h + \frac{ h^2}{2 \delta_2} \right)  |\hat{\omega}|^2   + \left( \frac{ h^2  \delta_2}{2 } + \frac{h^3}{2} \right)  |\Delta F^{\star}|^2 \right) \\
                    &+ e^{- \gamma h } \left( \vphantom{ \frac{h^3}{2}} (e^{- \gamma h } - e^{\gamma h} ) |\omega|^2 +  \frac{h}{\delta_3} |\hat{\omega}|^2 + \left( h \delta_3 + h^2 \right) |\Delta F^{\star}|^2 \right) \;.
                    \end{aligned}
\end{aligned} 
\end{equation}

By the co-coercivity property in  \eqref{eq:cocoercive} evaluated at $x_1=x+ \frac{h}{2} v_O$ and $x_2 = \tilde{x}+ \frac{h}{2} \tilde{v}_O$, it holds that  $|\Delta F^{\star}|^2 \le L \, \Phi^{\star}$.  Inserting this bound into \eqref{ieq:Delta_1} yields
\begin{align}
 \Delta \ &\le \ \rn{1} + \rn{2} + \rn{3} + \rn{4}\;,  \quad \text{where} \label{ieq:delta} \\
 \rn{1} \ &:= \   - \frac{1}{2} \beta h e^{-\frac{\gamma h}{2}} \Phi^{\star} -  \sinh(\gamma h) |\hat{\omega}|^2\;,  \nonumber  \\ 
 \rn{2} \ &:= \   \alpha  \left( \frac{L h^4}{4}  +   \frac{L h^3 \delta_1}{4}  - h^2 \right)  \Phi^{\star}\;,  \nonumber \\
  \rn{3}  \ &:= \ \left(  \frac{\beta e^{-\frac{\gamma h}{2}}}{2}  \left(   L h^3 + L h^2  \delta_2  \right) + e^{-\gamma h} L (h^2 + h \delta_3)   - \frac{1}{2} \beta h e^{-\gamma h/2} \right) \Phi^{\star}\;,  \nonumber \\
   \rn{4} \ &:= \ \left( \alpha \left( h^2  + \frac{h^3}{4 \delta_1} \right) + \beta e^{-\frac{\gamma h}{2}} \left(  h + \frac{h^2 }{2 \delta_2} \right) + e^{-\gamma h} \frac{h}{\delta_3}   - \sinh(\gamma h) \right) |\hat{\omega}|^2\;.    \nonumber 
\end{align} 
Below in \eqref{ieq:rn2}-\eqref{ieq:rn4}, we show that $\rn{2}$-$\rn{4}$ are non-positive, and hence, $\Delta \le \rn{1}$.  In particular, choosing $\delta_1 = h$ yields \begin{align} \label{ieq:rn2}
\rn{2}\ =\ \alpha  h^2 \left(  L h^2 \left( \frac{1}{4} + \frac{\delta_1}{2 h} \right)  - 1 \right) \Phi^{\star} \ \overset{\ref{A_h}}{\le}\ \alpha  h^2 \left(  \frac{1}{6} \left( \frac{1}{4} + \frac{1}{2} \right)  - 1 \right) \Phi^{\star} \ \le\ 0  \;.
\end{align}
Choosing $\delta_2 = 2 h $, $\delta_3 = 2 \gamma^{-1}$, and by definition of $\alpha$ and $\beta$ in \eqref{eq:tw}, \begin{align} 
   &  \rn{3}  \ \le \  \frac{\gamma h }{4} e^{-\gamma h/2} \left(\frac{3}{2}  L h^2 + 4 L \gamma^{-2} ( \gamma h + 2) - \frac{1}{2} \right) \Phi^{\star} \nonumber \\
    & \overset{\ref{A_h}}{\le} \, \frac{\gamma h }{4} e^{-\gamma h/2} \left( L \gamma^{-2} \left( \frac{3}{2} + 12 \right)  - \frac{1}{2} \right) \Phi^{\star}  \, \overset{\ref{A_h}}{\le} \, \frac{\gamma h }{4} e^{-\gamma h/2} \left( \frac{1}{36} \frac{27}{2}  - \frac{1}{2} \right) \Phi^{\star} \: \le \: 0 \;, \label{ieq:rn3}
\end{align}
as well as
\begin{align} 
   &  \rn{4} \  \overset{\eqref{ieq:alpha}}{\le} \  \gamma h   \left(\frac{3}{20} \gamma h \left(1 + \frac{h}{4 \delta_1} \right) + \frac{1}{4} \left( 1 + \frac{h}{2 \delta_2} \right) + \frac{\gamma^{-1}}{\delta_3} -1 \right) |\hat{\omega}|^2  \nonumber \\
  & \le\ \gamma h   \left(\frac{3}{20} \gamma h \left(1 + \frac{1}{4} \right) + \frac{1}{4} \left( 1 + \frac{1}{4} \right) -\frac{1}{2} \right) |\hat{\omega}|^2 \ \overset{\ref{A_h}}{\le} \ 0   \;. \label{ieq:rn4}
\end{align}

\medskip

Combining the above, and by definition of  $\beta$ in \eqref{eq:tw}, \begin{align}
\Delta \ \le \ \rn{1}
& \overset{ \ref{A_K}}{\le}   - \frac{1}{2} \beta h e^{-\frac{\gamma h}{2}} K |\zeta^{\star}|^2 -  \sinh(\gamma h) e^{-\gamma h} |\omega|^2  \nonumber \\  
& \overset{\eqref{ieq::exp}}{\le}  -\min\left(  K \gamma^{-1} h e^{-\frac{\gamma h}{2}}, \gamma h \right)   ( \frac{1}{8} \gamma^2 |\zeta^{\star}|^2 +  \frac{1}{2} |\omega|^2 )  \nonumber \\
& \ \le \   -\min\left(  K \gamma^{-1} h e^{-\frac{\gamma h}{2}}, \gamma h \right)   ( \frac{1}{16} \gamma^2 |\zeta|^2 +  \left( \frac{1}{2} - \frac{1}{16} \gamma^2 h^2 \right) |\omega|^2 )  \nonumber \\
& \overset{\ref{A_h}}{\le} - \frac{1}{16} e^{-1/2} \min\left(  K \gamma^{-1} h , \gamma h \right) ( \gamma^2 |\zeta|^2 +  |\omega|^2 )  \nonumber \\
& \overset{\eqref{eq:tw_equiv}}{\le} -  \frac{1}{17} e^{-1/2} \min\left(  K \gamma^{-1} , \gamma \right)  h \twnorm{(\zeta,\omega)}^2  \;.     \label{ieq:rn1} 
\end{align}
Thus, $ \twnorm{ Z - \tilde{Z} } \le (1-   c \, h) \twnorm{ z - \tilde{z} } $ with $c=(1/34) e^{-1/2} \min\left(  K \gamma^{-1} , \gamma \right) $ --- as required.
\end{proof}

\begin{comment}
By synchronously coupling the noise entering the OU steps of two copies of OABAO starting at different initial conditions, and applying Lemma~\ref{lem:OABAO_contr}, we get:

\begin{theorem}[$\W_{\tw}^2$-Contractivity of ukLa] \label{thm:contractivity}
Suppose that Assumptions~\ref{A_L}, ~\ref{A_K}, and~\ref{A_h} hold.  Then for any probability measures $\nu, \eta \in \PM^2(\mathbb{R}^{2 d})$, \begin{align}
\W_{\tw}^2(\nu \tilde{\Pi}^m, \eta \tilde{\Pi}^m) \, &\le \, (1 - c \, h)^m \, \W_{\tw}^2(\nu, \eta) \;.
\end{align}
\end{theorem}

By applying \eqref{eq:tw_equiv}, we obtain the following corollary.

\begin{corollary}[$\W^2$-Convergence of pukLa] 
Suppose that Assumptions~\ref{A_L}, ~\ref{A_K}, and~\ref{A_h} hold.  Then for any probability measures $\nu, \eta \in \PM^2(\mathbb{R}^{2 d})$, \begin{align}
\W^2(\nu \tilde{\Pi}^m, \eta \tilde{\Pi}^m) \, &\le \, \sqrt{2} \, \gamma \, e^{-\frac{1}{4}} \, (1 - c \, h)^m \, \W^2(\nu, \eta) \;.
\end{align}
\end{corollary}
\end{comment}

\subsection{Verifying \ref{A_S} {\it(ii)}: One-shot Coupling for UKLA}

By using a one-shot coupling $\Pi^u_{Reg}$, cf. \cite{roberts2002one,madras2010quantitative,EbMa2019,monmarche2021high, BouRabeeEberle2023}, the next lemma proves that the transition kernel of UKLA has a regularizing effect.  Throughout this section, for any $z=(x,v),\tilde{z}=(\tilde{x}, \tilde{v}) \in \mathbb{R}^{2d}$, $\mathsf{a}_1,\mathsf{a}_2 \in \mathbb{R}^d$, let $\Phi: \mathbb{R}^{2d} \to \mathbb{R}^{2d}$ be the \emph{one-shot} map $\Phi: (\mathsf{a}_1,\mathsf{a}_2) \mapsto (\tilde{\mathsf{a}}_1, \tilde{\mathsf{a}}_2)$ implicitly defined by  \begin{equation} \label{eq:Phi}
 O_{h/2}(\mathsf{a}_{2}) \circ \theta_h \circ O_{h/2}(\mathsf{a}_{1})(z) = 
O_{h/2}(\tilde{\mathsf{a}}_{2}) \circ \theta_h \circ O_{h/2}(\tilde{\mathsf{a}}_{1})(\tilde{z}) \;.
\end{equation}

\begin{lemma}  \label{lem:overlap_pukLa}
Suppose Assumptions~\ref{A_L},~\ref{A_LH}, and~\ref{A_h} hold.  Let $\xi_1, \xi_2 \sim \mathcal{N}(0,I_d)$ be independent.  Then, for all $z, \tilde{z} \in \mathbb{R}^{2d}$, it holds that \begin{align} 
\| \delta_z \Pi^u_{Reg} - \delta_{\tilde{z}} \Pi^u_{Reg} \|_{\TV} \ &\le \ 
\| \law(\xi_1,\xi_2) - \law(\Phi(\xi_1,\xi_2))\|_{\TV} \label{tvbound1} \\
 &\le \ \frac{7}{2} \, \left( \frac{1}{(\gamma h)^{3/2}} + d^{1/2} \frac{L_H h^3}{ \gamma h}  \right) \,  \| z - \tilde{z} \|\;.\label{tvbound2} 
\end{align}
\end{lemma}

A closely related one-shot coupling result has recently been developed for an ``OBABO'' discretization of kinetic Langevin dynamics; see Proposition 3 and Proposition 22 of \cite{monmarche2021high}, and for extensions see \cite{
monmarche2022entropic,camrud2023second}.  Although the upper bound in \eqref{tvbound2} degenerates as $h \searrow 0$, this degeneration manifests only logarithmically in the mixing time results for MAKLA; see Assumption~\ref{A_S} {\it(iii)}.  

\begin{proof}
Since the map $(\xi_1, \xi_2) \mapsto O_{h/2}(\xi_{2}) \circ \theta_h \circ O_{h/2}(\xi_{1})(z)$ is deterministic and measurable \cite[Lemma 3]{madras2010quantitative}, \begin{align*}
\| \delta_z \Pi^u_{Reg} - \delta_{\tilde{z}} \Pi^u_{Reg} \|_{\TV} \ \le \ \| \law(\xi_1,\xi_2)-\law(\Phi(\xi_1,\xi_2)) \|_{\TV} \;,
\end{align*} 
which gives \eqref{tvbound1}.   Inserting Lemmas~\ref{lem:overlap},~\ref{lem:oneshot:a}, and~\ref{lem:oneshot:b} into \eqref{tvbound1} gives \eqref{tvbound2}. 
\end{proof}

As already indicated, the following lemmas are used in the proof of Lemma~\ref{lem:overlap_pukLa}.

\begin{lemma} \label{lem:overlap}
Let $\xi \sim \mathcal{N}(0,\Sigma)$ where $\Sigma$ is an $n \times n$ matrix and suppose that $\mathsf{F}: \mathbb{R}^{n} \to \mathbb{R}^{n}$ is an invertible and differentiable map. Then
\begin{equation} \label{tvbound}
\begin{aligned}
& \| \law(\xi) - \law(\mathsf{F}(\xi))  \|_{\TV} \ \le \ \\
& \qquad \frac{1}{2} 
\sqrt{ \mathbb E\Bigr[ |\Sigma^{-\frac{1}{2}} (\mathsf{F}(\xi)-\xi)|^2 + 2 \tr(  D \mathsf{F}(\xi)-I_n ) - 2 \log|\det D \mathsf{F}(\xi)| \Bigr] } \;.
\end{aligned}
\end{equation}
\end{lemma}

\begin{proof}[Proof of Lemma~\ref{lem:overlap}]
The proof of this result is an extension of the  proof of Lemma~15 in \cite{BouRabeeEberle2023} to the case where the covariance matrix of the reference Gaussian measure is $\Sigma$.  Since this is a small modification, a proof is omitted.  
\end{proof}

\begin{lemma}\label{lem:oneshot:a}
Suppose Assumptions~\ref{A_L}, ~\ref{A_LH}, and~\ref{A_h} hold.  Then, for any $z=(x,v),\tilde{z}=(\tilde{x}, \tilde{v})\in \mathbb{R}^{2d}$ and $\mathsf{a}_1,\mathsf{a}_2 \in \mathbb{R}^d$, \begin{equation} \label{Phiv}
| \Phi(\mathsf{a}_1, \mathsf{a}_2)-(\mathsf{a}_1, \mathsf{a}_2) |^2 \ \le \  \frac{44}{\gamma^3 h^3} \ \|z-\tilde{z} \|^2  \;.
\end{equation}
\end{lemma}

\begin{lemma}\label{lem:oneshot:b}
Suppose Assumptions~\ref{A_L}, ~\ref{A_LH}, and~\ref{A_h} hold. Then, for any $z=(x,v),\tilde{z}=(\tilde{x}, \tilde{v})\in \mathbb{R}^{2d}$ and $\mathsf{a}_1,\mathsf{a}_2 \in \mathbb{R}^d$, \begin{equation} \label{DPhiv}
\tr( D \Phi(\mathsf{a}_1,\mathsf{a}_2 ) - I_{2d} )  -  \log |\det D\Phi(\mathsf{a}_1,\mathsf{a}_2 )|  \ \le \  2 d \frac{(L_H h^3)^2}{\gamma^2 h^2}    \|z-\tilde{z} \|^2  \;.   
\end{equation}
\end{lemma}

For the proofs of Lemmas~\ref{lem:oneshot:a} and~\ref{lem:oneshot:b}, it is notationally convenient to define  \begin{align*} 
& \zeta \, := \, x - \tilde{x} \;, \quad \omega \, := \, v - \tilde{v} \;, \\
& v_O \, := \, e^{-\frac{\gamma h}{2}} \ v + \sqrt{1-e^{-\gamma h}} \mathsf{a}_1 \;, \quad \tilde{v}_O \, := \,e^{-\frac{\gamma h}{2}} \ \tilde{v} +  \sqrt{1-e^{-\gamma h}} \tilde{\mathsf{a}}_1 \;, \quad   \\
& x^{\star} \, := \,  x + \frac{h}{2} v_O \;, \quad \tilde{x}^{\star} \, := \,  x  + \frac{h}{2} v_O \;, \quad  \zeta^{\star} \, := \, x^{\star} - \tilde{x}^{\star}    \;.   
\end{align*} 
This elementary inequality is used in the proofs: by~\ref{A_h} and \eqref{ieq::exp}, \begin{equation}
1 - e^{-\gamma h} \ge \frac{\gamma h}{2} \implies (1 - e^{-\gamma h})^{-1} \le \frac{2}{\gamma h} \;. 
\label{ieq:exp_2}
\end{equation}
Since $\zeta^{\star} = \zeta +  \frac{h}{2} e^{-\frac{\gamma h}{2}} \omega + \frac{h}{2}  \sqrt{1 - e^{-\gamma h}} ( \mathsf{a}_1 - \tilde{\mathsf{a}}_1) $, \begin{align}
| \zeta^{\star}|^2 \ \le \ 3 \left( |  \zeta|^2 + \frac{h^2}{4} | \omega |^2  + \frac{h^2}{4} (1-e^{-\gamma h}) | \mathsf{a}_1 - \tilde{\mathsf{a}}_1|^2 \right)  \;. \label{ieq:zstar_1}
\end{align}
By inserting \eqref{ieq:delta_a1} from the calculation below into \eqref{ieq:zstar_1}, and using the elementary inequality $1-e^{-\mathsf{x}} \le \mathsf{x}$ valid for $\mathsf{x} > -1$, we   obtain \begin{align}
| \zeta^{\star}|^2 \le 3 \left( |  \zeta|^2 +  \frac{h^2}{4} | \omega|^2  + \frac{\gamma h^3}{4}  | \mathsf{a}_1 - \tilde{\mathsf{a}}_1|^2 \right) \overset{\ref{A_h}}{\le}  8 \gamma^{-2} \left( \gamma^2 |  \zeta|^2 +  |  \omega|^2 \right)   \label{ieq:zstar_2} \;.
\end{align}

\begin{proof}[Proof of Lemma~\ref{lem:oneshot:a}]
By definition of the one-shot map in \eqref{eq:Phi}, \begin{align}
& | \mathsf{a}_1 - \tilde{\mathsf{a}}_1 |^2 = \frac{1}{1-e^{-\gamma h}} | \frac{1}{h}   \zeta + e^{-\frac{\gamma h}{2}} \omega - \frac{h}{2} H_U^{\star} \zeta^{\star} |^2 \nonumber \\
&\quad \ \le\  \frac{3}{1-e^{-\gamma h}}  \left( \frac{1}{h^2} | \zeta|^2 + | \omega|^2 + \frac{h^2}{4}  \mnorm{H_U^{\star}}^2 |\zeta^{\star} |^2 \right) \nonumber  \\
&\quad \overset{\ref{A_L}}{\le}  \frac{3}{1-e^{-\gamma h}}  \left( \frac{1}{h^2} | \zeta|^2 + | \omega|^2 + \frac{L^2 h^2}{4}  |\zeta^{\star} |^2 \right)  \nonumber  \\
&\quad \overset{\eqref{ieq:zstar_1},\eqref{ieq:exp_2}}{\le} \left( \frac{6}{\gamma^3 h^3} +  \frac{9}{2} L^2 h \gamma^{-3}   \right) \gamma^2 | \zeta |^2 + \left( \frac{6}{\gamma h} + \frac{9}{8}  L^2 h^3 \gamma^{-1} \right) |  \omega|^2  + \frac{9}{16}  L^2 h^4 | \mathsf{a}_1 - \tilde{\mathsf{a}}_1 |^2 \nonumber  \\
& \quad  \overset{\ref{A_h}}{\implies} | \mathsf{a}_1 - \tilde{\mathsf{a}}_1 |^2  \ \le \ \frac{181}{30 \gamma^3 h^3} \left( \gamma^2 | \zeta |^2 + | \omega|^2 \right) \label{ieq:delta_a1}
\end{align}

Similarly, from \eqref{eq:Phi},  \begin{align}
& |\mathsf{a}_2 - \tilde{\mathsf{a}}_2 |^2 = \frac{1}{1-e^{-\gamma h}} \left| e^{-\frac{\gamma h}{2}} \left(  e^{-\frac{\gamma h}{2}} \omega  + h \sqrt{1 - e^{-\gamma h}}  (\mathsf{a}_1 - \tilde{\mathsf{a}}_1) - h  H_U^{\star} \zeta^{\star} \right) \right|^2 \nonumber \\
& \le \frac{3}{1-e^{-\gamma h}} \left(  | \omega|^2 + (1-e^{-\gamma h}) |\mathsf{a}_1 - \tilde{\mathsf{a}}_1|^2 + h^2 |H_U^{\star}  \zeta^{\star} |^2 \right) \nonumber \\
& \overset{\ref{A_L}}{\le}   \frac{3}{1-e^{-\gamma h}} \left( L^2 h^2 |\zeta^{\star} |^2  + |\omega|^2 \right)  + 3 |\mathsf{a}_1 - \tilde{\mathsf{a}}_1|^2   \nonumber \\
& \overset{\eqref{ieq:zstar_2}}{\le}  \frac{3}{1-e^{-\gamma h}} \left( 8 L^2 h^2 \gamma^{-2}  \left( \gamma^2 | \zeta|^2 + |\omega|^2 \right)  + |\omega|^2 \right)  + 3 |\mathsf{a}_1 - \tilde{\mathsf{a}}_1|^2 
\nonumber \\
&  \overset{\eqref{ieq:exp_2}}{\le}  \frac{6}{\gamma h } \left( 8 L^2 h^2 \gamma^{-2}  \left( \gamma^2 | \zeta|^2 + |\omega|^2 \right)  + |\omega|^2 \right)  + 3 |\mathsf{a}_1 - \tilde{\mathsf{a}}_1|^2 
 \nonumber \\
 & \overset{\eqref{ieq:delta_a1}}{\le} \frac{6}{\gamma h } \left( 8 L^2 h^2 \gamma^{-2}  \left( \gamma^2 | \zeta|^2 + |\omega|^2 \right)  + |\omega|^2 \right)  +  \frac{181}{10 \gamma^3 h^3} \left( \gamma^2 | \zeta |^2 + | \omega|^2 \right)
 \nonumber \\
  & \overset{\ref{A_h}}{\le} \frac{25}{ \gamma^3 h^3} \left( \gamma^2 | \zeta |^2 + | \omega|^2 \right) \;.
 \label{ieq:delta_a2}
\end{align}

Combining \eqref{ieq:delta_a1} and \eqref{ieq:delta_a2} yields, \begin{align}
& | \mathsf{a}_1 - \tilde{\mathsf{a}}_1 |^2 + |\mathsf{a}_2 - \tilde{\mathsf{a}}_2 |^2 \le  \frac{44}{\gamma^3 h^3} \left( \gamma^2 | \zeta |^2 + | \omega|^2 \right) \nonumber
\end{align}
as required.
\end{proof}

\begin{proof}[Proof of Lemma~\ref{lem:oneshot:b}]
Since, by definition of the one-shot map in \eqref{eq:Phi},  \[
\frac{\partial \tilde{\mathsf{a}}_2}{\partial \mathsf{a}_2} = I_d \qquad \text{and} \qquad \frac{\partial \tilde{\mathsf{a}}_1}{\partial \mathsf{a}_2} = 0 \;,
\]  it follows that $D\Phi = \begin{pmatrix}
\frac{\partial \tilde{\mathsf{a}}_1}{\partial \mathsf{a}_1} &  0\\
\frac{\partial \tilde{\mathsf{a}}_2}{\partial \mathsf{a}_1} & I_d \end{pmatrix}$, and hence, \begin{align}
& \det D\Phi(\mathsf{a}_1,\mathsf{a}_2 ) = \det \frac{\partial \tilde{\mathsf{a}}_1}{\partial \mathsf{a}_1} \cdot \det \frac{\partial \tilde{\mathsf{a}}_2}{\partial \mathsf{a}_2} = \det \frac{\partial \tilde{\mathsf{a}}_1}{\partial \mathsf{a}_1} \;, \quad \text{and} \label{eq:detDPhi} \\
& \tr( D \Phi(\mathsf{a}_1,\mathsf{a}_2 ) - I_{2d} ) = \tr( \frac{\partial \tilde{\mathsf{a}}_1}{\partial \mathsf{a}_1}  - I_d ) \;. \label{eq:trDPhi}
\end{align} Combining \eqref{eq:detDPhi} and \eqref{eq:trDPhi} yields 
\begin{equation} \label{eq:trdeDPhi}
\begin{aligned}
&  \tr( D \Phi(\mathsf{a}_1,\mathsf{a}_2 ) - I_{2d} )  -  \log |\det D\Phi(\mathsf{a}_1,\mathsf{a}_2 )| \\
& \qquad = \tr( \frac{\partial \tilde{\mathsf{a}}_1}{\partial \mathsf{a}_1}  - I_d ) - \log |\det \frac{\partial \tilde{\mathsf{a}}_1}{\partial \mathsf{a}_1}| \;.
   \end{aligned}
\end{equation}
This observation motivates estimating 
$\mnorm{ \frac{\partial \tilde{\mathsf{a}}_1}{\partial \mathsf{a}_1} - I_d}^2 $.

From \eqref{eq:Phi}, note that 
\begin{align}
 & \frac{\partial \tilde{\mathsf{a}}_1}{\partial \mathsf{a}_1} - I_d = \frac{1}{\sqrt{1-e^{-\gamma h}}}
\frac{\partial}{\partial \mathsf{a}_1}  \left[ - \frac{h}{2}  (\nabla U(x^{\star}) - \nabla U(\tilde{x}^{\star}) ) \right] \nonumber \\
&=  - \frac{h^2}{4}  \left( D^2 U(x^{\star}) - D^2U(\tilde{x}^{\star}) \frac{\partial \tilde{\mathsf{a}}_1}{\partial \mathsf{a}_1} \right) \nonumber \\
&= - \frac{h^2}{4} \left( D^2 U(x^{\star}) - D^2U(\tilde{x}^{\star})  \right) + \frac{h^2}{4} D^2U(\tilde{x}^{\star}) \left( \frac{\partial \tilde{\mathsf{a}}_1}{\partial \mathsf{a}_1} - I_d \right) \;.  \label{eq:dta1}
\end{align}
On the one hand, by \ref{A_L},
\begin{align}
 & \mnorm{\frac{\partial \tilde{\mathsf{a}}_1}{\partial \mathsf{a}_1} - I_d}^2 \le  \frac{3 h^4}{16}  \sup_x \mnorm{D^2 U(x)}^2 \left(  2  +  \mnorm{\frac{\partial \tilde{\mathsf{a}}_1}{\partial \mathsf{a}_1} - I_d}^2 \right) \nonumber \\
& \overset{\ref{A_L}}{\le} \frac{3 (L h^2)^2}{16}   \left(  2  +  \mnorm{\frac{\partial \tilde{\mathsf{a}}_1}{\partial \mathsf{a}_1} - I_d}^2 \right)   \overset{\ref{A_h}}{\implies} \mnorm{\frac{\partial \tilde{\mathsf{a}}_1}{\partial \mathsf{a}_1} - I_d} \le \frac{1}{2} \;. \label{ieq:dta1}
\end{align}
On the other hand, by \ref{A_LH},
\begin{align}
 & \mnorm{\frac{\partial \tilde{\mathsf{a}}_1}{\partial \mathsf{a}_1} - I_d}^2 \overset{\ref{A_LH}}{\le}   \frac{h^4}{8} \left( L_H^2 |\zeta^{\star}|^2 + \sup_x \mnorm{D^2 U(x)}^2 \mnorm{\frac{\partial \tilde{\mathsf{a}}_1}{\partial \mathsf{a}_1} - I_d}^2 \right) \nonumber \\
 & \overset{\ref{A_L}}{\le}  \frac{1}{8} \left( L_H^2 h^4 |\zeta^{\star}|^2 + (L h^2)^2 \mnorm{\frac{\partial \tilde{\mathsf{a}}_1}{\partial \mathsf{a}_1} - I_d}^2 \right)     \nonumber \\
& \overset{\ref{A_h}}{\implies} 
 \mnorm{\frac{\partial \tilde{\mathsf{a}}_1}{\partial \mathsf{a}_1} - I_d}^2 \le \frac{1}{7}  L_H^2 h^4 |\zeta^{\star}|^2  \ \overset{\eqref{ieq:zstar_2}}{\le} \  \frac{8}{7} \frac{(L_H h^3)^2}{\gamma^2 h^2} \left( \gamma^2 |  \zeta|^2 +  |  \omega|^2 \right) \;.
 \label{ieq:dta2}
 \end{align}

Combining \eqref{ieq:dta1} and \eqref{ieq:dta2} yields \[
\mnorm{\frac{\partial \tilde{\mathsf{a}}_1}{\partial \mathsf{a}_1} - I_d} \le \min\left( \frac{1}{2},  \sqrt{\frac{8}{7}}  \frac{L_H h^3}{\gamma h} \| z - \tilde{z} \|  \right) \;. 
\]  Since $\mnorm{\frac{\partial \tilde{\mathsf{a}}_1}{\partial \mathsf{a}_1} - I_d} \le 1/2$, the spectral radius of $\frac{\partial \tilde{\mathsf{a}}_1}{\partial \mathsf{a}_1} - I_d$ does not exceed $1/2$.   Therefore, we can invoke Theorem 1.1 of \cite{rump2018estimates}, to obtain  \begin{align} 
& \tr( \frac{\partial \tilde{\mathsf{a}}_1}{\partial \mathsf{a}_1} - I_d ) - \log|\det \frac{\partial \tilde{\mathsf{a}}_1}{\partial \mathsf{a}_1} | \le \frac{\norm{\frac{\partial \tilde{\mathsf{a}}_1}{\partial \mathsf{a}_1} - I_d}_F^2 / 2}{1-\mnorm{\frac{\partial \tilde{\mathsf{a}}_1}{\partial \mathsf{a}_1} - I_d}} \nonumber  \\
& \qquad  \le  \norm{\frac{\partial \tilde{\mathsf{a}}_1}{\partial \mathsf{a}_1} - I_d}_F^2  \le d \mnorm{\frac{\partial \tilde{\mathsf{a}}_1}{\partial \mathsf{a}_1} - I_d}^2 \le 2  d  \frac{(L_H h^3)^2}{\gamma^2 h^2}  \|z-\tilde{z} \|^2   \;,  \label{eq:pertId} 
\end{align}
where in the second to last step we used $\norm{\frac{\partial \tilde{\mathsf{a}}_1}{\partial \mathsf{a}_1} - I_d}_F \le \sqrt{d} \mnorm{\frac{\partial \tilde{\mathsf{a}}_1}{\partial \mathsf{a}_1} - I_d}$.  
Inserting \eqref{eq:pertId} into \eqref{eq:trdeDPhi} gives the required result. 
\end{proof}

\subsection{Verifying \ref{A_S} {\it(iii)}: Energy Error Estimates}

The following Lemma provides upper bounds for the energy error in terms of the energy-like function $\mathcal{E}: \mathbb{R}^{2d} \to \mathbb{R}$ defined by 
\begin{equation}
    \label{eq:calE}
 \mathcal E(z)\ =\ |v|^2 + L^{-1}|\nabla U(x)|^2\;. \end{equation}
As the isotropic Gaussian case suggests, where $U(x)=(L/2)|x|^2$, the scaling in \eqref{eq:calE} is natural, since in that case: if $(X,V) \sim \mu$ then $L^{-1/2}\nabla U(X)=L^{1/2}X$ and $V$ are both standard normally distributed.

\begin{lemma}\label{lem:posVerlet_deltaH}
Suppose that Assumption~\ref{A_L} holds and let $Lh^2\leq1$.
Then, the energy error $\Delta H=H\circ\theta_h-H$ with $\theta_h$ as in~\eqref{eq:thetah} satisfies
\[ |\Delta H(z)|\ \leq\ 4Lh^2\mathcal E(z)\quad\text{for all $z\in\mathbb R^{2d}$.} \]
If additionally Assumption~\ref{A_LH} holds, then
\[ |\Delta H(z)|\ \leq\ 2L_Hh^3\mathcal E(z)^{3/2}+L^{3/2}h^3\mathcal E(z)\quad\text{for all $z\in\mathbb R^{2d}$.} \]
\end{lemma}

\begin{proof}
For $t \in [0, h]$, introduce the linear interpolation
\[ (x_t,v_t)\ =\ \,\bigr(x+tv-\frac{th}{2}\nabla U(x^*),\ v-t\nabla U(x^*)\bigr)\;, \]
where $x^*=x+\frac{h}{2}v$.  Note that $(x_0,v_0) = (x,v)$ and $(x_h,v_h) = \theta_h(x,v)$. Therefore, the energy error can be written as
\begin{equation} \Delta H(x,v)\ =\ (H\circ\theta_h-H)(x,v)\ =\ \int_0^h\frac{\mathrm d}{\mathrm dt}H(x_t,v_t)\dd t\;. 
\label{eq:deltaH}
\end{equation}
Expanding the integrand using $H(x,v)=|v|^2/2+U(x)$ yields
\begin{align} \nonumber
        &\frac{\mathrm d}{\mathrm dt}H(x_t,v_t)
\ =\   v_t\cdot\frac{\mathrm d}{\mathrm dt}v_t + \nabla U(x_t)\cdot\frac{\mathrm d}{\mathrm dt}x_t \\
&  =   -\bigr(v-t\nabla U(x^*)\bigr)\cdot\nabla U(x^*) + \nabla U\bigr(x+tv-\frac{th}{2}\nabla U(x^*)\bigr)\cdot\bigr(v-\frac{h}{2}\nabla U(x^*)\bigr)\;.\label{deltaHderivative}
\end{align}

Let $t\geq0$ and $a\in\mathbb R^d$.
To further simplify the last display, we expand
\[ \nabla U(x+ta)\ =\ \nabla U(x)+\int_0^t\nabla^2U(x+sa)a\dd s \]
and
\[ \nabla U(x+ta)\ =\ \nabla U(x)+t\nabla^2U(x)a+\int_0^t(t-s)\nabla^3U(x+sa):a^{\otimes2}\dd s\;. \]
In particular, 
\begin{equation}\label{star1}
    \nabla U(x^*)\ =\ \nabla U(x)+I^*_1
\end{equation}
with $I^*_1=\frac{1}{2}\int_0^h\nabla^2U\bigr(x+\frac{s}{2}v\bigr)v\dd s$ satisfying $|I^*_1|\leq\frac{1}{2}Lh|v|$, and
\begin{equation}\label{star2}
    \nabla U(x^*)\ =\ \nabla U(x)+\frac{h}{2}\nabla^2U(x)v+I^*_2
\end{equation}
with $I^*_2=\frac{1}{4}\int_0^h(h-s)\nabla^3U\bigr(x+\frac{s}{2}v\bigr):v^{\otimes2}\dd s$ satisfying $|I^*_2|\leq\frac{1}{8}L_Hh^2|v|^2$.
Further,
\begin{equation}\label{taylor1}
    \nabla U\bigr(x+tv-\frac{th}{2}\nabla U(x^*)\bigr)\ =\ \nabla U(x) + I_1(t)
\end{equation}
with $I_1(t)=\int_0^t\nabla^2U\bigr(x+sv-\frac{sh}{2}\nabla U(x^*)\bigr)\bigr(v-\frac{h}{2}\nabla U(x^*)\bigr)\dd s$ satisfying $|I_1(t)|\leq Lt\bigr|v-\frac{h}{2}\nabla U(x^*)\bigr|$, and
\begin{equation}\label{taylor2}
    \nabla U\bigr(x+tv-\frac{th}{2}\nabla U(x^*)\bigr)\ =\ \nabla U(x)+t\nabla^2U(x)\bigr(v-\frac{h}{2}\nabla U(x^*)\bigr)+I_2(t)
\end{equation}
with $I_2(t)=\int_0^t(t-s)\nabla^3U\bigr(x+sv-\frac{sh}{2}\nabla U(x^*)\bigr):\:\bigr(v-\frac{h}{2}\nabla U(x^*)\bigr)^{\otimes2}\dd s$ satisfying $|I_2(t)|\leq\frac{1}{2}L_Ht^2\bigr|v-\frac{h}{2}\nabla U(x^*)\bigr|^2$.

Using the higher and lower order expansions will give the higher and lower order bound, respectively, due to more or less cancellation.
For the lower order bound, inserting \eqref{star1} and \eqref{taylor1}  into \eqref{deltaHderivative} yields
\begin{align*}
        \frac{\mathrm d}{\mathrm dt}H(x_t,v_t)
\ =\   \begin{aligned}[t]
        &(t-h/2)|\nabla U(x)|^2-I_1^*\cdot\bigr(v+\frac{1}{2}(h-4t)\nabla U(x)\bigr)+t|I_1^*|^2\\
        &+I_1(t)\cdot\bigr(v-\frac{h}{2}\nabla U(x^*)\bigr)\;.
        \end{aligned}
\end{align*}
Inserting this expression and the bounds on $I_1^*$ and $I_1(t)$ back into \eqref{eq:deltaH} shows
\begin{align*}
        |\Delta H(x,v)|
\ &\leq\ \frac{1}{2}Lh^2|v|\bigr|v-\frac{h}{2}\nabla U(x)\bigr|+\frac{1}{8}L^2h^4|v|^2+\frac{1}{2}Lh^2\bigr|v-\frac{h}{2}\nabla U(x^*)\bigr|^2 \\
&\leq\  4Lh^2\mathcal E(z)\;,
\end{align*}
where, besides $Lh^2\leq1$, we used that due to \eqref{star1}
\[ \bigr|v-\frac{h}{2}\nabla U(x^*)\bigr|\ \leq\ |v|+\frac{h}{2}|\nabla U(x^*)|\ \overset{\eqref{star1}}{\leq}\ \bigr(1+\frac{1}{4}Lh^2\bigr)|v|+\frac{h}{2}|\nabla U(x)|\ \leq\ 2\mathcal E(z)^{1/2} \]
and that a similar bound holds for $\bigr|v-\frac{h}{2}\nabla U(x)\bigr|$.  Repeating the calculation with the higher order expansions \eqref{star2} and \eqref{taylor2}  gives
\begin{align*}
        &|\Delta H(x,v)|
\ \ \leq\ \begin{aligned}[t]
        &-h\bigr(v-\frac{h}{2}\nabla U(x^*)\bigr)\cdot I_2^*-\frac{h^3}{4}v\cdot\nabla^2U(x)\nabla U(x^*)\\
        &+\frac{h^4}{8}\nabla U(x^*)\cdot\nabla^2U(x)\nabla U(x^*)+\int_0^hI_2(t)\dd t\cdot\bigr(v-\frac{h}{2}\nabla U(x^*)\bigr)
        \end{aligned} \\
& \quad \leq\  \bigr(\frac{4}{3}+\frac{1}{4}\bigr)L_Hh^3\mathcal E(z)^{3/2}+\frac{1}{4}Lh^3\bigr(|v|+\frac{h}{2}|\nabla U(x^*)|\bigr)|\nabla U(x^*)| \\
& \quad \leq\  2L_Hh^3\mathcal E(z)^{3/2}+L^{3/2}h^3\mathcal E(z)\;,
\end{align*}
as required.
\end{proof}

\subsection{Verifying \ref{A_S} {\it(iv)}: Exit Probability Estimates for MAKLA}

We now turn to the exit probability bound required in Assumption \ref{A_S} {\it(iv)}.
For some suitably large $R_U$ and $\mathcal E$ as in \eqref{eq:calE}, we show that the exit probability from
\[ D\ =\ \,\bigr\{\mathcal E(z)\leq R_U\bigr\} \]
is small over the total number of steps $\mathfrak H$ required to attain the desired $\TV$ convergence.
More precisely, let $(Z_k)_{k\geq0}$ be a copy of MAKLA started in an initial distribution $\nu$ and define the first exit time of the chain from $D$ to be
\[ T\ =\ \inf\bigr\{k\geq0\,:\,Z_k\notin D\bigr\}\;. \]
The following lemma is general in the sense that it only assumes an energy error bound satisfied by, amongst other discretizations, $\theta_h$ as in~\eqref{eq:thetah}.

\begin{lemma}\label{lem:exit}
Suppose Assumptions~\ref{A_K} and~\ref{A_L} hold and that the energy error $\Delta H=H\circ\theta_h-H$ satisfies
\begin{equation}\label{stabdeltaH}
    |\Delta H(z)|\ \leq\ C_{\Delta H}h^k\mathcal E(z)
\end{equation}
for some $C_{\Delta H}>0$, $k\geq2$ and all $z\in\mathbb R^{2d}$.
Let $h,\gamma>0$ be such that
\begin{equation}\label{exit_h}
    (1+25C_{\Delta H}h^k)^2\max(\gamma h,1)\ \leq\ 4\;.
\end{equation}
Then, for $\mathfrak{H}, R_U>0$, it holds that
\begin{equation}\label{exit_bd}
    \mathbb P\bigr(T\leq\mathfrak H \bigr)\ \leq\  \exp\Bigr(\frac{1+25C_{\Delta H}h^k}{4}\gamma h \mathfrak H d-\frac{1-50C_{\Delta H}\mathfrak Hh^{k}}{16}R_U\Bigr)\nu\bigr(e^{H/8}\bigr)\;.
\end{equation}
If additionally $100C_{\Delta H}\mathfrak Hh^{k}\leq1$, then $\mathbb P(T\leq\mathfrak H)\leq\eps/4$ for $\eps>0$ if
\begin{equation}\label{exit_R}
    R_U\ \geq\ 32\bigr[\gamma h \mathfrak Hd+\log\bigr(4\nu(e^{H/8})/\eps\bigr)\bigr]\;.
\end{equation}
\end{lemma}

Although the energy error bound \eqref{stabdeltaH} is assumed to hold globally for simplicity, it can be relaxed to hold in a neighborhood of $D$; cf.~Remark~\ref{rmk:broad}.

\begin{remark}[Effect of Velocity Flip]\label{rmk:flip}
The MAKLA transition step involves a velocity flip involution in the event of a rejection; cf. \eqref{eq:OMABAO}.
This makes it tricky to construct a Foster-Lyapunov function exploiting the contractivity of the unadjusted kernel by incorporating a cross-term $x \cdot v$.
The function used below does not involve such a cross-term, and as a consequence, the time horizon $h \mathfrak H$ enters linearly into \eqref{exit_R}.
In contrast, similar bounds for the MALA transition step only require the radius to depend logarithmically on the time horizon  \cite[\S 6]{Eb2014}.
However, due to the wide availability of energy error bounds such as \eqref{stabdeltaH}, the Foster-Lyapunov function presented here is a robust alternative.
\end{remark}

%since it is only applied to $O_{h/2}(\xi)(z)$ for $z\in D$, which is contained in a neighborhood of $D$ with high probability.

\begin{proof}
Below, we will show that the Lyapunov function $e^{H/8}$ solves 
\begin{equation} \label{eq:bvp}
    \begin{cases}
        \ \mathcal L_he^{H(z)/8}\ \leq\ \left( e^{\lambda}-1 \right) e^{H(z)/8}    &\text{for }z\in D , \\
        \ e^{H(z)/8}\ \geq\ e^{R_U/16}                  &\text{for }z\in\partial D,
        \end{cases}
\end{equation}
where $\mathcal L_h=\pi-\mathrm{id}$ is the generator of MAKLA and \[ \lambda \ = \ \frac{1}{8}[2(1+25C_{\Delta H}h^k)\gamma hd+25C_{\Delta H}h^kR_U] \;. 
\]
By Chernoff's inequality, 
\begin{align*}
        &\mathbb P\bigr(T\leq\mathfrak H\bigr)  \leq\  \mathbb E\exp\Bigr(- \lambda (T-\mathfrak H)\Bigr) \leq\  \exp\Bigr( \lambda \mathfrak{H} -R_U/16\Bigr)\nu\bigr(e^{H/8}\bigr) \\
& \quad \leq\  \exp\Bigr(\frac{1}{4}(1+25C_{\Delta H}h^k)\gamma h \mathfrak H d-\frac{1}{16}(1-50C_{\Delta H}\mathfrak H h^{k})R_U\Bigr)\nu\bigr(e^{H/8}\bigr)\;,
\end{align*}
where in the second to last step we used a maximum principle to upper bound the Laplace transform of the first exit time $T$ by the solution $e^{H/8}$ of the boundary value problem in \eqref{eq:bvp}, i.e., $\mathbb{E} e^{-\lambda T}  e^{R_U/16} \le \nu(e^{H/8})$; for details see, e.g.,    \cite{EberleLectureNotes2023}. Therefore, if $100C_{\Delta H}\mathfrak Hh^{k}\leq1$, it follows that $\mathbb P(T\leq\mathfrak H)\leq\eps/4$ if
\[ R_U\ \geq\ 32\bigr[\gamma h \mathfrak Hd+\log\bigr(4\nu(e^{H/8})/\eps\bigr)\bigr]\;. \]

\medskip

We now turn to the proof that $e^{H/8}$ solves \eqref{eq:bvp}.
On $\partial D$, the lower bound holds since $U(x)\geq\frac{1}{2L}|\nabla U(x)|^2$ by \ref{A_L}, and hence,
\[ e^{H(z)/8}\ =\ e^{(|v|^2/2+U(x))/8}\ \geq\ e^{\mathcal E(z)/16}\ =\ e^{R_U/16}\;. \]
Let $z=(x,v)\in D$ and set $\delta=1/8$. Then
\[ \mathcal L_he^{H(z)/8}\ =\ \Bigr[\mathbb E\exp\Bigr(\delta\bigr[H\circ O_{h/2}(\xi_2)\circ\hat\theta_h(\mathcal U)\circ O_{h/2}(\xi_1)(z)-H(z)\bigr]\Bigr)-1\Bigr]e^{H(z)/8} \;. \]
Therefore, the upper bound in \eqref{eq:bvp}  holds if
\begin{equation}\label{stabts}
\begin{aligned}
        &\mathbb E\exp\Bigr(\delta\bigr[H\circ O_{h/2}(\xi_2)\circ\hat\theta_h(\mathcal U)\circ O_{h/2}(\xi_1)(z)-H(z)\bigr]\Bigr) \\
&\quad \leq\  \exp\Bigr(\delta\bigr[2\bigr(1+(1+3/\delta)C_{\Delta H}h^k\bigr)\gamma hd+(1+3/\delta)C_{\Delta H}h^kR_U\bigr]\Bigr)\ =\ e^{\lambda} \;.
\end{aligned}
\end{equation}
Treating the two outcomes of the Metropolis step separately yields 
\begin{align}\label{stabgoal}
        \mathbb E & \exp\Bigr(\delta\bigr[H\circ O_{h/2}(\xi_2)\circ\hat\theta_h(\mathcal U)\circ O_{h/2}(\xi_1)(z)-H(z)\bigr]\Bigr) \overset{\eqref{eq:OMABAO}}{\leq} \rn{1} + \rn{2} \quad \text{where} \\  \rn{1} &= \mathbb E\exp\Bigr(\delta\bigr[H\circ O_{h/2}(\xi_2)\circ\theta_h\circ O_{h/2}(\xi_1)(z)-H(z)\bigr]\Bigr) \nonumber \\
        \rn{2} &= \mathbb E\Bigr[
    \exp\Bigr(\delta\bigr[H\circ O_{h/2}(\xi_2)\circ\mathcal S
        \circ O_{h/2}(\xi_1)(z)-H(z)\bigr]\Bigr)\,;\,
        A\bigr(O_{h/2}(\xi_1)(z)\bigr)^c\Bigr]\;.  \nonumber
\end{align} To bound $\rn{1}$ in \eqref{stabgoal},  we use the elementary bound
\begin{equation}\label{elem:int}
    \mathbb E_{\xi\sim\mathcal N(0,I_d)}\exp\bigr(b\cdot\xi+c|\xi|^2\bigr)\ \leq\ \exp\bigr(|b|^2+2cd\bigr)
\end{equation}
valid for $c\in [0,1/4]$ and $b\in\mathbb R^d$.  In particular, applying \eqref{elem:int} with $c=\gamma h \delta$ which is possible since $c \le 1/4$ by \eqref{exit_h}, implies that for all $z'=(x',v')\in\mathbb R^{2d}$ and for $\xi\sim\mathcal N(0,I_d)$  
\begin{align} 
        &\mathbb E\exp\Bigr(\delta\bigr[H\circ O_{h/2}(\xi)(z')-H(z')\bigr]\Bigr) \nonumber \\
&=\     \mathbb E\exp\Bigr(\delta\bigr[-\frac{1}{2}(1-e^{-\gamma h})|v'|^2+e^{-\gamma h/2}\sqrt{1-e^{-\gamma h}}v'\cdot\xi+\frac{1}{2}(1-e^{-\gamma h})|\xi|^2\bigr]\Bigr) \nonumber\\
&\overset{\eqref{elem:int}}{\leq}\  \exp\Bigr(\delta\bigr[-\frac{1}{2}(1-2\delta e^{-\gamma h})(1-e^{-\gamma h})|v'|^2+(1-e^{-\gamma h})d\bigr]\Bigr)
\ \leq\  e^{\delta\gamma hd}\;. \label{stabgreenstar}
\end{align}
Therefore, using $\Delta H=H\circ\theta_h-H$, 
\begin{align}
& \rn{1} \ = \     \mathbb E\exp\Bigr(\delta\bigr[
        \begin{aligned}[t]
        &H\circ O_{h/2}(\xi_2)\circ\theta_h\circ O_{h/2}(\xi_1)(z)-H\circ\theta_h\circ O_{h/2}(\xi_1)(z)\\
        &+\Delta H\circ O_{h/2}(\xi_1)(z)+H\circ O_{h/2}(\xi_1)(z)-H(z)\bigr]\Bigr)
        \end{aligned} \nonumber\\
& \overset{\eqref{stabgreenstar}}{\leq} \  e^{\delta\gamma hd}\mathbb E\exp\Bigr(\delta\bigr[\Delta H\circ O_{h/2}(\xi_1)(z)+H\circ O_{h/2}(\xi_1)(z)-H(z)\bigr]\Bigr)\;.\label{prestabrn1}
\end{align}
Inserting the energy error bound \eqref{stabdeltaH} into the exponent of \eqref{prestabrn1} then yields,
\begin{align*}
        &\Delta H\circ O_{h/2}(\xi_1)(z)+H\circ O_{h/2}(\xi_1)(z)-H(z) \\
&\leq\   C_{\Delta H}h^k\mathcal E(O_{h/2}(\xi_1)(z))+H\circ O_{h/2}(\xi_1)(z)-H(z) \\
&\leq\   \begin{aligned}[t]
            &C_{\Delta H}h^k\mathcal E(z)-\frac{1}{2}(1-e^{-\gamma h})|v|^2+(1+2C_{\Delta H}h^k)e^{-\gamma h/2}\sqrt{1-e^{-\gamma h}}v\cdot\xi_1\\
            &+\frac{1}{2}(1+2C_{\Delta H}h^k)\gamma h|\xi_1|^2\;.
        \end{aligned}
\end{align*}
Hence, inserting this bound back into \eqref{prestabrn1}, using $\mathcal E(z)\leq R_U$, and  applying \eqref{elem:int} with $c=\delta (1+2C_{\Delta H}h^k) \gamma h /2$ which is possible since $c\leq1/4$ holds by \eqref{exit_h}, shows \begin{align}
 \rn{1} \ &\le \ \exp\Bigr(\delta\bigr[
    \begin{aligned}[t]
    &2(1+C_{\Delta H}h^k)\gamma hd+C_{\Delta H}h^kR_U \nonumber \\
    &-\frac{1}{2}\bigr(1-2\delta(1+2C_{\Delta H}h^k)^2e^{-\gamma h}\bigr)(1-e^{-\gamma h})|v|^2\bigr]\Bigr) \nonumber \\
\end{aligned} \nonumber \\
 \ &\le \ \exp\Bigr(\delta\bigr[2(1+C_{\Delta H}h^k)\gamma hd+C_{\Delta H}h^kR_U\bigr]\Bigr)  \label{stabrn1}
\end{align}
since $2\delta(1+2C_{\Delta H}h^k)^2\leq1$ by \eqref{exit_h}. For $\rn{2}$ in \eqref{stabgoal},  using $H\circ\mathcal S=H$,
\begin{align*}
        &\rn{2}\ =\     \mathbb E\Bigr[
        \begin{aligned}[t]
        &\mathbb E_{\xi_2}\exp\Bigr(\delta\bigr[H\circ O_{h/2}(\xi_2)\bigr(\mathcal S
        \circ O_{h/2}(\xi_1)(z)\bigr)-H\bigr(\mathcal S\circ O_{h/2}(\xi_1)(z)\bigr)\bigr]\Bigr)\\
        &\times\exp\Bigr(\delta\bigr[H(O_{h/2}(\xi_1)(z))-H(z)\bigr]\Bigr)\,;\,
        A\bigr(O_{h/2}(\xi_1)(z)\bigr)^c\Bigr]
        \end{aligned}  \\
& \quad \overset{\eqref{stabgreenstar}}{\leq} \  e^{\delta\gamma hd}\mathbb E\Bigr[\exp\Bigr(\delta\bigr[H(O_{h/2}(\xi_1)(z))-H(z)\bigr]\Bigr)\,;\,
        A\bigr(O_{h/2}(\xi_1)(z)\bigr)^c\Bigr]\;.
\end{align*}
We continue estimating the last display using Cauchy-Schwarz inequality combined with $\mathbb P(A(z)^c)=1-e^{\Delta H(z)^+}\leq|\Delta H(z)|$,  \eqref{stabgreenstar}, and \eqref{stabdeltaH} to obtain
\begin{align}
\nonumber        &\rn{2}  \ \le \ e^{\delta\gamma hd}\Bigr(\mathbb Ee^{2\delta[H(O_{h/2}(\xi_1)(z))-H(z)]}\Bigr)^{1/2}\bigr(\mathbb E|\Delta H(O_{h/2}(\xi_1)(z))|^2\bigr)^{1/2} \\
&\overset{\eqref{stabgreenstar}}{\leq}   e^{2\delta\gamma hd}C_{\Delta H}h^k\bigr(\mathbb E\,\mathcal E(O_{h/2}(\xi_1)(z))^2\bigr)^{1/2}
 \leq  3e^{2\delta\gamma hd}C_{\Delta H}h^k(R_U+2\gamma hd)\;, \label{stabrn2}
\end{align}
where in the last step we used
\[ \mathbb E\,\mathcal E(O_{h/2}(\xi_1)(z))^2\ \leq\ 4\mathbb E\bigr(\mathcal E(z)+\gamma h|\xi_1|^2\bigr)^2\ \leq\ 8\bigr(R_U^2+3(\gamma h d)^2\bigr)
%\ =\ e^{-\gamma h}|v|^2+L^{-1}|\nabla U(x)|^2+(1-e^{-\gamma h})\mathbb E|\xi_1|^2\ \leq\ R_U+\gamma hd
\;. \]
Inserting \eqref{stabrn1} and \eqref{stabrn2} into \eqref{stabgoal} and simplifying yields
\begin{align}
        &\mathbb E\exp\Bigr(\delta\bigr[H\circ O_{h/2}(\xi_2)\circ\hat\theta_h(\mathcal U)\circ O_{h/2}(\xi_1)(z)-H(z)\bigr]\Bigr) \nonumber \\
&\leq\  \exp\Bigr(\delta\bigr[2(1+C_{\Delta H}h^k)\gamma hd+C_{\Delta H}h^kR_U\bigr]\Bigr) + 3e^{2\delta\gamma hd}C_{\Delta H}h^k(R_U+2\gamma hd) \nonumber \\
&\leq\  \exp\Bigr(\delta\bigr[2(1+C_{\Delta H}h^k)\gamma hd+C_{\Delta H}h^kR_U\bigr]\Bigr)\cdot\bigr(1+3C_{\Delta H}h^k(R_U+2\gamma hd)\bigr) \nonumber \\
&\leq\  \exp\Bigr(\delta\bigr[2(1+C_{\Delta H}h^k)\gamma hd+C_{\Delta H}h^kR_U\bigr]+\log\bigr(1+3C_{\Delta H}h^k(R_U+2\gamma hd)\bigr)\Bigr) \nonumber \\
&\leq\  \exp\Bigr(\delta\bigr[2\bigr(1+(1+3/\delta)C_{\Delta H}h^k\bigr)\gamma hd+(1+3/\delta)C_{\Delta H}h^kR_U\bigr]\Bigr)\ =\ e^{\lambda}  \;, \nonumber
\end{align}
where we used $\log(1+ \mathsf{x})\leq \mathsf{x}$ valid for $\mathsf{x}\geq0$.
This proves \eqref{stabts} holds, and hence, $\exp(H(z)/8)$ solves \eqref{eq:bvp}, as required.
\end{proof}

\printbibliography

\end{document}